\documentclass[12pt]{article}
\usepackage{amsmath,amsfonts,amssymb,amsthm}
\usepackage{xcolor}
\usepackage{color}
\AtBeginDocument{
	\label{CorrectFirstPageLabel}
	
}

\makeatletter
\newcommand\xleftrightarrow[2][]{%
	\ext@arrow 9999{\longleftrightarrowfill@}{#1}{#2}}
\newcommand\longleftrightarrowfill@{%
	\arrowfill@\leftarrow\relbar\rightarrow}
\makeatother

\newtheorem{proposition}{Proposition}[section]
\newtheorem{theorem}[proposition]{Theorem}
\newtheorem{corollary}[proposition]{Corollary}
\newtheorem{lemma}[proposition]{Lemma}
\newtheorem{remark}[proposition]{Remark}

\newtheorem{example}[proposition]{Example}

\newcommand{\nc}{\newcommand}
\nc{\I}{{\mathbf 1}}

\nc{\bN}{{\mathbf N}}
\nc{\bM}{{\mathbf M}}
\nc{\cB}{{\mathcal B}}
\nc{\cM}{{\mathcal M}}
\nc{\R}{{\mathbb R}}
\nc{\N}{{\mathbb N}}
\nc{\Z}{{\mathbb Z}}
\nc{\BX}{{\mathbb X}}
\nc{\BY}{{\mathbb Y}}
\nc{\cX}{{\mathcal X}}
\nc{\cY}{{\mathcal Y}}
\nc{\cN}{{\mathcal N}}
\nc{\cF}{{\mathcal F}}

\nc{\tv}{\mathbf{d_{\mathrm {TV}}}}
\nc{\kr}{\mathbf{d_{\mathrm {KR}}}}

\nc{\BP}{\mathbb{P}}
\nc{\BE}{\mathbb{E}}
\nc{\BQ}{\mathbb{Q}}

\numberwithin{equation}{section}

\begin{document} 

\renewcommand{\thefootnote}{\fnsymbol{footnote}}
\author{{\sc G\"unter Last\footnotemark[1]}\, and {\sc Moritz Otto\footnotemark[2]}}
\footnotetext[1]{guenter.last@kit.edu,  
Karlsruhe Institute of Technology, Institute of Stochastics,
76131 Karlsruhe, Germany. }
\footnotetext[2]{otto@math.au.dk, 
Aarhus University, Department of Mathematics, Ny Munkegade 118, 8000 Aarhus C, Denmark. }

\title{Disagreement coupling of  Gibbs processes\\
with an application to Poisson approximation} 
\date{\today}
\maketitle

\begin{abstract} 
\noindent 
We discuss a thinning and an embedding procedure to
construct finite Gibbs processes with a given
Papangelou intensity. Extending the approach in \cite{HT19,HTHou17} 
we will use this to couple two finite Gibbs processes with different
boundary conditions. As one application we will establish
Poisson approximation of point processes derived
from certain infinite volume Gibbs processes via dependent thinning.
As another application we shall discuss empty space probabilities
of certain Gibbs processes.

\end{abstract}

\noindent
{\bf Keywords:} Gibbs process, disagreement coupling, Poisson approximation,
Papangelou intensity, Poisson thinning, Poisson embedding, empty space probabilities

\vspace{0.1cm}
\noindent
{\bf AMS MSC 2010:} 60G55, 60D05, 60K35

\section{Introduction}\label{sintro}

In the seminal paper \cite{BergMaes94} the authors 
introduced {\em disagreement percolation} for discrete
Markov random fields to control boundary effects and 
to establish new uniqueness criteria for Gibbs measures on graphs. 
The idea is to locally couple two fields with different
boundary conditions and to control the disagreement
with a stochastically dominating percolation process. 
In \cite{HT19,HTHou17} the method was developed for Gibbs processes
in continuum. The first aim of this paper is to establish
{\em disagreement coupling} in great generality  using modern
point process theory. This way we shall also close a gap
left in \cite{HT19,HTHou17}. The second main aim is to combine this
coupling with a recent result from \cite{BSY21}
(generalizing a classical result from \cite{BB92})
to obtain Poisson approximation of point processes 
derived of Gibbsian functionals.
We shall also discuss empty space probabilities of Gibbs processes.

Gibbs processes form an important class of point processes.
In (mathematical) physics they describe the thermodynamical behavior 
of interacting particles; see \cite{Ruelle70}. As a mathematical model they are much more
versatile than the completely independent Poisson process
(see e.g.\ \cite{Dereudre18}). They are also quite popular
in spatial statistics; see e.g.\ \cite{MoeWaa07,CSKM13}. 
Getting a handle on distributional properties of Gibbs processes is not easy.
However, starting with the seminal paper \cite{SchrYuk13}, the last decade
has seen some efforts to understand some asymptotic
properties of functionals of Gibbs processes in infinite volume, at least
if the process is in a certain sense
close to a Poisson process. Our paper aims at adding
to this development.

Let us shortly summarize the structure and main results of the paper.
Section \ref{secGibbs} collects some basic facts on Gibbs processes based on
the classical point process approach from \cite{NgZe79} and
\cite{MaWaMe79}. The state space is assumed to be Borel but is
otherwise not required to have any topological properties.  The main
result of Section \ref{secthinning} is the thinning representation
from Theorem \ref{t2.1} which extends a result in
\cite{HTHou17} to general Borel spaces. Our proof is based on Lebesgue--Stieltjes calculus and
is different from the one given in \cite{HTHou17}.  In
Section \ref{secembedding} we prove with Theorem \ref{tembedding} a
version of the thinning representation based on embedding into a
Poisson process on a suitable product space.  This result is quite
convenient for coupling purposes and should be compared with the
classical Poisson embedding of marked point processes based on
stochastic intensities; see \cite{BreMass96}.  This embedding does not
require any boundedness assumptions on the Papangelou intensity.  

In Section \ref{secdisagreement} we formulate and study the disagreement
coupling of two (local) Gibbs processes with the same Papangelou
intensity but different boundary conditions. 
To this end we consider a
symmetric and measurable relation on the state space equipping each
point configuration with a graph structure.  At each point of the
state space the Papangelou intensity is then assumed to depend only on
the cluster connected to this point.
Theorem \ref{tdis} generalizes Theorem 3.1 in
\cite{HTHou17} and shows that the points of disagreement of two
recursively defined point processes are connected to the boundary
conditions. The main part of the proof (which seems to be missing in
\cite{HTHou17}) is devoted to checking that the 
coupled processes have the desired Gibbs distribution. 
Our main tool here is a basically well-known spatial Markov
property of the Poisson process. We work here
with a rather general definition of a stopping
set given in the Appendix of \cite{LPY2021}.
We have chosen the terminology ``disagreement coupling''
as opposed to ``disagreement percolation'' because this coupling
might be potentially useful also beyond a percolation setting.
In fact, Theorem \ref{tdis} does not require the
absence of percolation in the dominating Poisson process.
But indeed, so far all current applications of disagreement coupling
require the absence of percolation. The main example
are proofs of uniqueness of certain Gibbs distributions; 
see \cite{HT19,BergMaes94,HTHou17}. 
A very recent result in this area
is \cite{BetschLast22}, a paper that makes crucial use of 
the general setting in Theorem \ref{tdis}.
In \cite{BHLV20} disagreement coupling was used to
establish exponential decorrelation of certain Gibbs processes.  

In Section \ref{sempty} we discuss empty space probabilities
of Gibbs processes. In particular we show that a large class
of Gibbs processes are Poisson-like as defined in \cite{SchrYuk13}.
This property is needed in the proof of Theorem \ref{hardtpo}
but we believe that it is of importance in its own right.

Section \ref{secPalm} contains some (basically well-known) material on Palm distributions of
a Gibbs process which is needed later. In Theorem \ref{tpoapprox} we
prove a bound on the total variation distance between an appropriately
scaled thinning of a marked Gibbs process in $\mathbb{R}^d$ with
bounded Papangelou intensity and a Poisson process. Our proof adapts a
coupling technique for Poisson approximation from
\cite[Theorem 3.3]{Otto20} to Gibbs processes. 
In doing so it exploits the
disagreement coupling studied in Section \ref{secdisagreement} for a
Gibbs process and its Palm version.
Similarly as in \cite{BergMaes94,HTHou17} we assume that
a random graph defined on the points of a dominating Poisson
process does not percolate. In fact we need an exponentially
small cluster size; see \eqref{csharp}.
This condition appears to  be weaker than subcriticality
of an associated spatial branching type process, as assumed
in \cite{SchrYuk13}; see \cite{BHLV20} for a short discussion.
Our general bound involves expectations with respect to
the Gibbs distribution. This cannot be avoided.
Still, apart from subcriticality of the dominating percolation model, 
our total variation bound is more explicit and more general than the one
presented in \cite{Schuh09}.
We demonstrate this in Section \ref{sMatern}  by applying
 Theorem \ref{tpoapprox} to Mat\'{e}rn
type I thinnings of Gibbs processes. In this situation
the bound from Theorem \ref{tpoapprox} becomes very
concrete. 


\section{Gibbs processes}\label{secGibbs}

Let $(\BX,\cX)$ be a Borel space equipped with a
$\sigma$-finite measure $\lambda$ and define
the set ring $\cX_0:=\{B\in\cX:\lambda(B)<\infty\}$.
Let $\bN\equiv \bN(\BX)$ be the space of all measures on $\BX$
which are $\N_0$-valued on $\cX_0$ (where $\N_0:=\N\cup\{0\}=\{0,1,2,\ldots\}$) and
let $\cN\equiv\cN(\BX)$ denote the smallest $\sigma$-field
such that $\mu\mapsto \mu(B)$ is measurable for
all $B\in\cX$. A {\em point process} is a random
element $\eta$ of $\bN$, defined over some fixed
probability space $(\Omega,\cF,\BP)$.
The {\em intensity measure} of $\eta$ is the
measure $\BE[\eta]$ defined by $\BE[\eta](B):=\BE\eta(B)$, $B\in\cX$.
For any $\sigma$-finite measure $\nu$ on $\BX$ we let
$\Pi_\nu$ denote the distribution of a {\em Poisson process} 
(see e.g.\ \cite{LastPenrose17}) on $\BX$
with intensity measure $\nu$.  Of particular interest for
us are the distributions $\Pi_{\alpha\lambda_B}$, where $\alpha\ge 0$
and $\lambda_B:=\lambda(B\cap\cdot)$ is the restriction of
$\lambda$ to some $B\in\cX_0$. We use the latter notation for any
measure on $\BX$. For each $B\in\cX$ we define
$\bN_B:=\{\mu\in\bN:\mu(B^c)=0\}=\{\mu_B:\mu\in\bN\}$.
For $\mu \in \bN$ we write $x \in \mu$ if $\mu(\{x\})>0$.
If $\mu(\{x\})\in\{0,1\}$ for all $x\in\BX$ then
$\mu$ is called {\em simple}.

Let $\kappa\colon\BX\times\bN\to\R_+$ be measurable.
A point process $\xi$ on $\BX$ is called a {\em Gibbs process}
with {\em Papangelou intensity} (PI) $\kappa$ if
\begin{align}\label{eGNZ}
\BE\Bigg[\int f(x,\xi)\,\xi(\mathrm{d}x)\Bigg]=
\BE\bigg[ \int f(x,\xi+\delta_x) \kappa(x,\xi)\,\lambda(\mathrm{d}x)\bigg],
\end{align}
for each measurable $f\colon\BX\times\bN\to\R_+$.
The latter are the {\em GNZ equations} named after Georgii, Nguyen and Zessin
\cite{Georgii76,NgZe79}.

In the present generality it is not known, whether Gibbs processes
exist. Partial answers (under varying assumptions on $\kappa$) are given 
for instance in \cite{Ruelle70,Mase00,DeGeDro12,DVass19,Jansen19}.
A considerably more challenging topic is the distributional uniqueness of 
a Gibbs process. From a mathematical point of view not much is known in this
regard. We refer to \cite{Dereudre18,Jansen19} for recent surveys.

Equation \eqref{eGNZ} can be generalized. For $m\in\N$ we define
a measurable function $\kappa_m\colon \BX^m\times\bN\to \R_+$ by
\begin{align*}
 \kappa_m(x_1,\ldots,x_m,\mu)
:=\kappa(x_1,\mu)\kappa(x_2,\mu+\delta_{x_1})\cdots\kappa(x_m,\mu+\delta_{x_1}+\cdots+\delta_{x_{m-1}}).
\end{align*}
Note that $\kappa_1=\kappa$.
If $\xi$ is a Gibbs process with PI $\kappa$ then
we have for each $m\in\N$ and  for each measurable $f\colon\BX^m\times\bN\to\R_+$ that
\begin{align}\label{eGNZmulti}
\BE&\bigg[\int f(x_1,\ldots,x_m,\xi)\,\xi^{(m)}(\mathrm{d}(x_1,\ldots,x_m))\bigg]\\ \notag
&=\BE \bigg[\int f(x_1,\ldots,x_m,\xi+\delta_{x_1}+\cdots+\delta_{x_m}) \kappa_m(x_1,\ldots,x_m,\xi)
\,\lambda^m(\mathrm{d}(x_1,\ldots,x_m))\bigg],
\end{align}
where $\mu^{(m)}$ is the $m$-th factorial measure of $\mu$; see \cite{LastPenrose17}.
This follows by induction, using that
\begin{align*}
\kappa_{m+1}(x_1,\ldots,x_{m+1},\mu)=\kappa_m(x_1,\ldots,x_m,\mu)
\kappa(x_{m+1},\mu+\delta_{x_1}+\cdots+\delta_{x_m}). 
\end{align*}

For each $m\in\N$ let $\tilde\kappa_m\colon \BX^m\times\bN\to \R_+$ 
be the symmetrization of $\kappa_m$ in the first $m$ arguments.
The {\em Hamiltonian}
$H\colon\bN\times\bN\to (-\infty,\infty]$ 
(based on $\kappa$) is defined by
\begin{align}\label{eHamilton}
 H(\mu,\psi):=
\begin{cases}
0,&\text{if $\mu(\BX)=0$}, \\
 -\log \tilde\kappa_m(x_1,\ldots,x_m,\psi),&
 \text{if $\mu=\delta_{x_1}+\cdots+\delta_{x_m}$}, \\
 \infty,&\text{if $\mu(\BX)=\infty$}.
 \end{cases}
\end{align}
For $B\in\cX_0$ the {\em partition function}
$Z_B\colon\bN\to[0,\infty]$ is defined by
\begin{align}\label{epartition}
 Z_B(\psi):=\int e^{-H(\mu,\psi)}\,\Pi_{\lambda_B}(\mathrm{d}\mu),\quad \psi\in\bN.
\end{align}
Since $H(0,\psi)=0$ for all finite $\psi\in\bN$ we have that
\begin{align*}
Z_B(\psi)\ge e^{-\lambda(B)}, \quad B\in\cX_0.
\end{align*}
Note that the right-hand side is positive.
For $\nu\in\bN$
the {\em Gibbs measure} $\Pi_{B,\nu}$ on $\bN$ is defined by
\begin{align}\label{eGibbsmeasure}
\Pi_{B,\nu}:=Z_B(\nu)^{-1}\int \I\{\mu\in\cdot\}e^{-H(\mu,\nu)}\,\Pi_{\lambda_B}(\mathrm{d}\mu)
\end{align}
provided that $Z_B(\nu)<\infty$. If $Z_B(\nu)=\infty$ we set $\Pi_{B,\nu}:=\Pi_{\lambda_B}$.
This measure is concentrated on $\bN_B$.

It was proved in \cite{MaWaMe79,NgZe79} that
if $\xi$ is a Gibbs process with PI $\kappa$ then
$\kappa_m(\cdot,\xi)$ is almost surely symmetric for each $m\in\N$,
\begin{align}\label{e2.3}
\BP(Z_B(\xi_{B^c})<\infty)=1,\quad B\in\cX_0,
\end{align}
and, for each measurable $f\colon\bN\to\R_+$,
\begin{align}\label{edlr}
 \BE[f(\xi_B)\mid \xi_{B^c}]
  =\int f(\mu)\,\Pi_{B,\xi_{B^c}}(\mathrm{d}\mu),
 \quad B\in\cX_0,
\end{align}
where relations involving conditional expectations are assumed
to hold almost surely.
These are the {\em DLR-equations}; see~\cite{Ruelle70,Kallenberg,Mase00}.
Note that \eqref{edlr} implies
\begin{align}\label{eempty}
 \BP(\xi(B)=0\mid \xi_{B^c})=e^{-\lambda(B)}Z_B(\xi_{B^c})^{-1}, \quad B\in\cX_0.
\end{align}

A natural requirement on $\kappa$  is the {\em cocycle formula}
\begin{align}\label{ecocycle}
\kappa(x,\mu)\kappa(y,\mu+\delta_x)=\kappa(y,\mu)\kappa(x,\mu+\delta_y),
\end{align}
which should hold at least
for  $\lambda^2\otimes\Pi_\lambda$-a.e.\ $(x,y,\mu)$.
(Otherwise one cannot hope for the existence of a Gibbs process.)
It then follows for each $m\in\N$ that
\begin{align}\label{esymmetric}
\kappa_m(x_1,\ldots,x_m,\mu)=\tilde\kappa_m(x_1,\ldots,x_m,\mu),\quad
\end{align}
for $\lambda^m\otimes\Pi_\lambda$-a.e.\ $(x_1,\ldots,x_m,\mu)$.
Even though not stated in the present generality, the following result was proved
in the seminal work \cite{NgZe79}.

\begin{theorem}\label{tNguyenZessin} Assume that $\kappa$ satisfies
\eqref{ecocycle} for each $\mu\in\bN$ and for $\lambda^2$-a.e.\ $(x,y)$.
Assume that $\xi$ is a point process satisfying \eqref{e2.3}
and \eqref{edlr}. Then $\xi$ is a Gibbs process with PI $\kappa$.
\end{theorem}

In fact, the proof of Theorem \ref{tNguyenZessin} shows the following.

\begin{corollary}\label{c2.3} Assume that $\lambda(\BX)<\infty$.
Assume also that $\kappa$ satisfies the assumption of Theorem \ref{tNguyenZessin}
and that $Z_\BX:=\int e^{-H(\mu,0)}\,\Pi_{\lambda}(\mathrm{d}\mu)<\infty$.
Then
\begin{align}\label{edlr2}
Z_{\BX}^{-1}\int\I\{\mu\in\cdot\}e^{-H(\mu,0)}\,\Pi_{\lambda}(\mathrm{d}\mu)
\end{align}
is the distribution of a Gibbs process with PI $\kappa$.
\end{corollary}

For $B\in\cX_0$ and $\psi\in\bN$ we
define  $\kappa_\psi\colon\BX\times\bN\to\R_+$ by
\begin{align}\label{ekappapsi}
\kappa_\psi(x,\mu):=\kappa(x,\psi+\mu)
\end{align}  
Let $\kappa_{B,\psi}$ denote the restriction of $\kappa_\psi$ to $B\times \bN_B$.
Corollary \ref{c2.3} shows that
the conditional distribution $\BP(\xi_B\in\cdot \mid \xi_{B^c})$
is almost surely a Gibbs process with PI  $\kappa_{B,\xi_{B^c}}$;
see also \cite[Lemma 2.5]{BHLV20}.

In Section \ref{secdisagreement} we shall need the following property of Gibbs measures.
Let $B,C\in\cX_0$ with $C\subset B$ 
and suppose that $\nu\in N_{\BX\setminus B}$ satisfies  $Z_B(\nu)<\infty$.
Then $Z_{B\setminus C}(\nu+\mu_C)<\infty$ for $\Pi_{B,\nu}$-a.e.\ $\mu$
and
\begin{align}\label{condDLR}
\Pi_{B,\nu}=\iint \I\{\mu_{C}+\mu'\in\cdot\}
\,\Pi_{B\setminus C,\nu+\mu_{C}}(\mathrm{d}\mu')\,\Pi_{B,\nu}(\mathrm{d}\mu).
\end{align}
This is the DLR-equation for a Gibbs process
with PI $\kappa_{B,\nu}$. It can be proved directly, using
the definition \eqref{eGibbsmeasure} of a Gibbs measure and the fact that
a Poisson process on $B$ is the sum of two independent
Poisson processes on $C$ and $B\setminus C$.

Finally we introduce some (partially standard) notation used later in the paper.
Let $\nu$ be a signed measure on some measurable
space $(\BY,\cY)$ which is locally finite in a suitable sense.
For instance $\nu$ could be the difference of two measures
on $\BX$ which are finite on $\cX_0$. 
Then $\nu$ can be uniquely written as the difference of two 
mutually singular locally finite measure, whose sum is the 
total variation measure $|\nu|$ of $\nu$. The total mass
$\|\nu\|:=|\nu|(\BY)$ (which could be infinite)
is the {\em total variation norm} of $\nu$.
If $\mu,\mu'\in\bN(\BX)$ then also $|\mu-\mu'|\in\bN(\BX)$.
If $\mu,\mu'$ are locally finite measure on $\BY$, then we can define
as in \cite{BSY21}
\begin{align}\label{e2.141}
\tv(\mu,\mu'):=\sup\{|\mu(A)-\mu'(A)|:A\in\cY,\mu(A)<\infty,\mu'(A)<\infty\}.
\end{align}
It is not hard to show that
\begin{align*}
\tv(\mu,\mu')\le \|\mu-\mu'\|\le 2\tv(\mu,\mu');
\end{align*}
see \cite[(2.2)]{BSY21} for the case $\mu,\mu'\in\bN$.
If $\mu(\BY)=\mu'(\BY)<\infty$ (and in particular if $\mu,\mu'$ are probability measures) 
then the second inequality above becomes an identity.

\section{Assumptions and Examples}\label{secass}

In this paper a measurable function
$\kappa\colon\BX\times\bN\to\R_+$ will always denote
a PI of a Gibbs process.
To simplify the presentation we shall always assume (sometime without further
mentioning) that $\kappa$ satisfies the cocycle
identity \eqref{ecocycle} for all $(x,y,\mu)\in\BX^2\times\bN$.
Very much as in the literature we will make two types
of assumptions on $\kappa$, namely stability 
assumptions and assumptions on the local dependence of $\kappa(x,\mu)$ on
$\mu$.
We say that $\psi\in\bN$ is {\em stable on $W\in\cX$} if 
\begin{align}\tag{Dom1}
Z_B(\psi_{B^c})<\infty,\quad B\in\cX_0\cap W,
\end{align}
where $\cX_0\cap W:=\{B\cap W:B\in\cX_0\}.$
If this holds for $W=\BX$ then we say that 
$\psi\in\bN$ is {\em stable}. 
We say that $\kappa$ is {\em locally stable}
if there exists a measurable function
$\alpha\colon\BX\to(0,\infty)$ satisfying
\begin{align}\tag{Dom2}
\kappa(x,\mu)\le \alpha(x),\quad (x,\mu)\in\BX\times\bN(\BX),
\quad\text{and}\quad
\int_B \alpha(x)\,\lambda(\mathrm{d}x)<\infty,\quad B\in\cX_0.
\end{align}

\begin{remark}\label{rbounded}\rm A sufficient condition for $\psi\in\bN$
to be stable on $W\in\cX_0$ is as follows. Suppose that
$\alpha_{W,\psi}\colon W\to\R_+$ is measurable and satisfies
$\int_W \alpha_{W,\psi}(x)\,\lambda(\mathrm{d}x)<\infty$.
If
\begin{align}
\sup\{\kappa(x,\mu_W+\psi_{W^c}):\mu\in\bN\}\le \alpha_{W,\psi}(x),\quad x\in W,\,\mu\in\bN,
\end{align}
then the definition of $Z_B$ shows that (Dom1) holds. In particular, if $\kappa$ is locally
stable, then all $\psi\in\bN$ are stable.
\end{remark}

To state our second type of assumptions we assume
that $\sim$ is a symmetric relation on $\BX$ such
that $\{(x,y):x\sim y\}$ is a measurable subset of $\BX^2$.
Given $x\in\BX$ and $A\subset \BX$ we write $x\sim A$ if there exists
$z\in A$ such that $x\sim z$.
We say that $x,z\in \BX$ are {\em connected} via $A$ if there
exist $n\in\N_0$ and $z_1,\ldots,z_n\in A$ such that
$z_i\sim z_{i+1}$ for $i\in\{0,\ldots,n\}$ where $z_0:=x$ and $z_{n+1}:=z$.
We use this terminology also for counting measures instead of $A$.
Define the {\em cluster}  $C(x,\mu)\in\bN$ of $x\in\BX$ in $\mu\in\bN$ by
\begin{align*}
C(x,\mu):=\int\I\{y\in\cdot\}\I\{\text{$y\ne x$ and $x$ and $y$ are connected via $\mu$}\}\,\mu(\mathrm{d}y).
\end{align*}
We will consider the assumption
\begin{align}\tag{Loc1}
\kappa(x,\mu)=\kappa(x,C(x,\mu)),\quad (x,\mu)\in\BX\times\bN.
\end{align}
For $x\in\BX$ we let $N_x:=\{y\in\BX \setminus \{x\}:y\sim x\}$.
A stronger version of (Loc1) is
\begin{align}\tag{Loc2}
\kappa(x,\mu)=\kappa(x,\mu_{N_x}),\quad (x,\mu)\in\BX\times\bN.
\end{align}
Of course, in such a general setting the assumptions (Loc1) and (Loc2)
do not put a restriction on $\kappa$. It is the specific choice of $\sim$ together
with further assumptions on $\lambda$, which will make them meaningful.

The following examples will illustrate the forgoing assumptions.

\begin{example} \rm \textbf{(Gibbs processes with pair potential)} 
	Suppose that $U$ is a {\em pair potential} on $\BX$ that
	is a measurable and symmetric function
	$U\colon\BX\times\BX\to(\infty,\infty]$. 
	We assume that there exists $C>0$ such that
	\begin{align*}
		\sum_{i<j}U(x_i,x_j)\ge -Cn,\quad x_1,\ldots,x_n\in\BX,\, n\in\N.
	\end{align*}
	This is a classical stability assumption on a pair potential; see \cite{Ruelle70}.
	Let $\alpha>0$ and assume that $\kappa$ is given by
	\begin{align*}
		\kappa(x,\mu)=\alpha \exp\bigg[-\int U(x,y)\,\mu(\mathrm{d}y)\bigg],\quad 
		x \in \BX,\,\mu \in \mathbf{N},
	\end{align*}
	whenever the integral in the right-hand side exists. 
	Let $W\in\cX_0$ and $\psi\in\bN$. A simple calculation shows
	that $Z_W(\psi_{W^c})<\infty$ if
	\begin{align}\label{eZS}
		\int_W\exp\bigg[-\int U(x,y)\,\psi(\mathrm{d}y)\bigg]\,\lambda(\mathrm{d}x)<\infty.
	\end{align}
	Assume moreover that $U$ is bounded from below
	and let $\psi\in \bN$ satisfy \eqref{eZS}. Then it is easy to see 
	that $\psi$ is stable on $W$ and also that $\psi$ is a regular boundary
	condition for $W$. 
	Define $x \sim y \Leftrightarrow U(x,y)\ne 0$.
	Then (Loc2) holds. If $U\ge 0$, then (Dom2) holds.
		
	The existence of Gibbs processes with a pair potential
	was shown under varying assumptions. We refer here to
	the seminal work \cite{Ruelle69,Ruelle70} (treating the case $\BX=\R^d$),
	to \cite{Jansen19} (working on a complete separable
	metric space and assuming $U\ge 0$) and
	to \cite{DVass19} (working on $\R^d$ but allowing for
	infinite range pair potentials which may take negative values).
\end{example}

\begin{example}\rm \textbf{(Strauss process)} \label{exstrauss}
	Suppose that $\mathbb{X}=\R^d \times \R_+$ and
	$\lambda=\lambda_d \otimes \mathbb{Q}$ for some probability measure
	$\mathbb{Q}$ on $[0,\infty)$. Define $(x,r)\sim (y,s) \Leftrightarrow \|x-y\|\le r+s$. Consider for
	$\alpha>0$ and $\beta \in [0,1]$ the PI
	\begin{align*}
		\kappa(x,r,\mu):=\alpha \beta^{\mu(N_{(x,r)})},\quad x \in \R^d, r>0, \mu \in \mathbf{N}.
	\end{align*}
	A Gibbs process $\xi$ with this PI is called Strauss
	process. For $\beta=0$ this is  a hard-core process. It
	is easy to see that $\kappa$ satisfies (Dom2) and (Loc2).
\end{example}

\begin{example}\rm \textbf{(Area interaction process)} 
	As for the Strauss process, let $\mathbb{X}:=\R^d \times \R_+$ and
	$\lambda:=\lambda_d \otimes \mathbb{Q}$ for some probability measure
	$\mathbb{Q}$ on $[0,\infty)$. Define $(x,r)\sim (y,s) \Leftrightarrow \|x-y\|\le r+s$.
	For $(x,r)\in\BX$ and $\mu\in\bN$ define 
	\begin{align*}
		V(x,r,\mu):=\lambda_d\bigg(B(x,r)\setminus\bigcup_{(y,s)\in \mu_{N_{(x,r)}}}B(y,s)\bigg).
	\end{align*}
	Consider for $\alpha>0$ and $\beta \in [0,1]$ the PI
	\begin{align*}
		\kappa(x,r,\mu):=\alpha \beta^{V(x,r,\mu)},\quad x \in \R^d, r>0, \mu \in \mathbf{N}.
	\end{align*}
	Then (Loc2) holds.
	Gibbs processes with this PI are called area interaction
	processes. Their
	existence is shown in \cite{Ruelle69} for $\mathbb{Q}$ with finite support and in 
	\cite{Dereudre09} for dimension $d\le 2$.
	If $\mathbb{Q}([r_1,\infty))=1$ for some $r_1>0$, then
	(Dom2) is satisfied.
\end{example}

\begin{example}\rm \textbf{(Continuum random cluster model)}
	Let $\mathbb{X}:=\R^d \times \R_+$ and
	$\lambda:=\lambda_d \otimes \mathbb{Q}$ for some probability measure
	$\mathbb{Q}$ on $[0,\infty)$. Define
	$(x,r)\sim (y,s) \Leftrightarrow \|x-y\|\le r+s$. Let $\alpha>0$,
	$q>0$ and $k(x,r,\mu)$ denote the number of connected components in
	$\mu$ connected to $(x,r)$ in $\mu+\delta_{(x,r)}$ and define
	\begin{align*}
		\kappa(x,r,\mu)=\alpha q^{1-k(x,r,\mu)},\quad x \in \R^d, r>0, \mu \in \mathbf{N}.
	\end{align*}
	It is clear that (Loc1) holds.   
	If $q\ge 1$ then (Dom2) holds. Moreover it was shown in \cite{DeHou15}
	that in both cases there exists a Gibbs process
	$\xi$ with PI $\kappa$. If there are $0<r_1<r_2<\infty$ such that
	$\mathbb{Q}([r_1,r_2])=1$, then there exists a constant $c_d$ (depending on
	the dimension $d$) such that $\kappa(x,r,\mu)\le q^{1-c_d}$ and
	(Dom2) holds.
\end{example}

\begin{example} \rm \textbf{(Widom-Rowlinson model)} For $m \ge 2$ let
	$\mathbb{M}:=\{1,\dots,m\}$,
	$\mathbb{X}:=\R^d \times \R_+ \times \mathbb{M}$ and
	$\lambda:=\lambda_d \otimes \mathbb{Q} \otimes \mathbb{U}$, where
	$\mathbb{Q}$ is a probability measure on $[0,\infty)$ and
	$\mathbb{U}$ denotes uniform distribution on $\mathbb{M}$. Define
	\begin{align*}
		(x,r,\ell)\sim (y,s,k)\quad \Longleftrightarrow \quad \|x-y\|\le r+s\,\text{ and }\ell\neq k,\quad x,y \in \mathbb{R}^d,\,r,s\in \R_+,\,\ell,k \in \mathbb{M}.
	\end{align*}
	and let
	\begin{align*}
		\kappa(x,r,\ell,\mu)=
		\begin{cases}
			0,& \text{if $\mu(N_{(x,r,\ell)})>0$},\\
			\alpha,&\text{otherwise}.
		\end{cases}
	\end{align*}
	If we think of the elements of $\mathbb{M}$ as colors, the
	construction rule of this model forbids overlapping of two balls of
	different colors. It was first introduced in \cite{WidRow70} for
	deterministic radii. For its existence we refer to \cite[Remark
	4.2]{GeorHae96}. Clearly, (Dom2) and (Loc2) are satisfied.
\end{example}

\section{Thinning presentation of finite Gibbs processes}\label{secthinning}

In this section we assume that the measure $\lambda$ is finite and diffuse.
We consider a measurable function $\kappa\colon\BX\times\bN\to\R_+$
satisfying the stability assumption (Dom2).
We know from \cite{GeorKun97} (see \cite{BHLV20} for the infinite case)
that a Gibbs process with PI $\kappa$ is stochastically
dominated by a Poisson process with intensity measure $\alpha\lambda$,
defined by $\alpha\lambda(\mathrm{d}x):=\alpha(x)\lambda(\mathrm{d}x)$.
In this section we shall reestablish this result
by means of a rather explicit thinning construction, introduced in 
\cite[Proposition 4.1]{HTHou17}.
(The proof was amended in the latest preprint version of that paper.) 
We work here in greater generality 
using fundamental properties of Poisson processes
and standard techniques from Lebesgue--Stieltjes calculus.

We need to introduce some notation. By definition of a Borel space there
exists an injective measurable mapping $\varphi\colon\BX\to\R$
such that $\varphi(\BX)$ and the inverse mapping 
$\varphi^{-1}\colon\varphi(\BX)\to\BX$ are measurable.
We introduce a total order $\le$ on $\BX$ by writing
$x\le y$ if $\varphi(x)\le\varphi(y)$. We write
$x< y$ if $x\le y$ and $x\ne y$. As usual we can then
define $(x,y]:=\{z\in\BX:x<z\le y\}$. Other intervals
are defined analogously. For instance we write
$(-\infty,y):=\{z\in\BX:z< y\}$.
Let $\bN^*\equiv \bN^*(\mathbb{X})$ 
denote the set of all simple and finite elements
of $\bN$. 
It is easy to see that the mapping
$(x,\mu)\mapsto (\mu_{(-\infty,x)},\mu_{(-\infty,x]})$ is measurable
on $\BX\times\bN^*$. We abbreviate $\mu_x:=\mu_{(-\infty,x)}$.

Assume that $\kappa\colon\BX\times\bN^*\to\R_+$ 
is a measurable function satisfying the cocycle assumption
\eqref{ecocycle} for all $(x,y,\mu)\in\BX^2\times\bN^*$.
We define the partition functions by \eqref{epartition}.
Define a function $p\colon \BX\times\bN^*\to[0,1]$ by
\begin{align}\label{ethprob}
p(x,\psi):=\kappa(x,\psi_x)
\frac{Z_{(x,\infty)}(\psi_x+\delta_x)}{Z_{(x,\infty)}(\psi_x)},\quad (x,\psi)\in\BX\times\bN^*,
\end{align}
where $\infty/\infty:=0$. Since
\begin{align*}
Z_{(x,\infty)}(\psi_x+\delta_x)
&=\int e^{-H(\mu_{(x,\infty)},\psi+\delta_x)}\,\Pi_{\lambda}(\mathrm{d}\mu),\\
Z_{[x,\infty)}(\psi_x)
&=\int e^{-H(\mu_{(x,\infty)},\psi)}\,\Pi_{\lambda}(\mathrm{d}\mu),
\end{align*}
the function $p$ is measurable. It turns out that it also satisfies
(Dom2).

\begin{lemma}\label{l4.1} We have that
\begin{align}\label{edom}
p(x,\psi)\le \alpha(x), \quad (x,\psi)\in \BX\times\bN^*.
\end{align}
\end{lemma}
{\em Proof:} Let $x \in \mathbb{X}$ and $\psi \in \bN^*$.
By the multivariate Mecke equation (\cite[Theorem 4.4]{LastPenrose17}),
\begin{align*}
\kappa(x,&\psi_x) Z_{(x,\infty)}(\psi_x+\delta_x)
=\kappa(x,\psi_x) \int e^{-H(\mu,\psi_x+\delta_x)}\,\mathrm{d}\Pi_{\lambda_{(x,\infty)}}(\mathrm{d}\mu)\\
  &=\kappa(x,\psi_x) e^{-\lambda((x,\infty))}\left(1+\sum_{m=1}^\infty \frac{1}{m!}
\int_{\mathbb{X}^m} \kappa_m(x_1,\dots,x_m,\psi_x+\delta_x)\,\lambda_{(x,\infty)}^m(\mathrm{d}(x_1,\dots,x_m))\right)\\
  &=e^{-\lambda((x,\infty))}\left(\kappa(x,\psi_x) + \sum_{m=1}^\infty \frac{1}{m!}
\int_{\mathbb{X}^m} \kappa_{m+1}(x_1,\dots,x_m,x,\psi_x)
\,\lambda_{(x,\infty)}^m(\mathrm{d}(x_1,\dots,x_m))\right).
\end{align*}
Since
$$
\kappa_{m+1}(x_1,\dots,x_m,x,\psi_x)=\kappa(x,\psi_x+\delta_{x_1}+\cdots+\delta_{x_m})
\kappa_m(x_1,\dots,x_m,\psi_x)
$$
and $\kappa(x,\cdot)\le \alpha(x)$ by (Dom2), the above is
bounded by
\begin{align*}
  &\alpha(x) e^{-\lambda((x,\infty))}\left(1+\sum_{m=1}^\infty \frac{1}{m!}
\int_{\mathbb{X}^m} \kappa_m(x_1,\dots,x_m,\psi_x)\,\lambda_{(x,\infty)}^m(\mathrm{d}(x_1,\dots,x_m))\right)
\end{align*}
which equals $\alpha(x) Z_{(x,\infty)}(\psi_x)$. This proves \eqref{edom}.
\qed

\bigskip

Let us define a kernel $K_{\kappa,\alpha}$ from $\bN^*$ to $\bN^*$ by
\begin{align}
  K_{\kappa,\alpha}(\mu,\cdot)
:=\sum_{\psi\le\mu}\I\{\psi\in\cdot\}\prod_{x\in\psi} \alpha(x)^{-1}p(x,\psi)
\prod_{x\in\mu-\psi}\big(1-\alpha(x)^{-1}p(x,\psi)\big),
\quad \mu\in\bN^*.
\end{align}
The next lemma shows that $K_{\kappa,\alpha}$ is a probability kernel.

\begin{lemma} We have that $K_{\kappa,\alpha}(\mu,\mathbf{N}^*)=1$ for all
$\mu \in \mathbf{N}^*$.
\end{lemma}
{\em Proof:} We use induction on the number of point of $\mu$.
Obviously, $K_{\kappa,\alpha}(0,\mathbf{N}^*)=1$, where
$0$ is the zero measure.

Let $n \in \N$ and $x_1,\dots,x_n,y \in \mathbb{X}$ with
$x_1<\cdots<x_n <y$. Define $\mu:=\delta_{x_1}+\dots+\delta_{x_n}$.
From the definition of $K_{\kappa,\alpha}$ we obtain that
we have
\begin{align*}
K_{\kappa,\alpha}(\mu+\delta_{y},\mathbf{N}^*)
&=\sum_{\psi \le \mu+\delta_y} 
\prod_{x \in \psi} \alpha(x)^{-1} p(x,\psi) \prod_{x \in \mu+\delta_y-\psi}
(1-\alpha(x)^{-1} p(x,\psi))\\
&= \sum_{\psi \le \mu} \Bigg(\prod_{x \in \psi+\delta_y} \alpha(x)^{-1} p(x,\psi+\delta_y) 
\prod_{x\in \mu-\psi}(1-\alpha(x)^{-1} p(x,\psi+\delta_y))\\
&\quad\quad \quad  \quad \quad
+\prod_{x \in \psi} \alpha(x)^{-1} p(x,\psi) \prod_{x \in \mu+\delta_y-\psi}
\left(1-\alpha(x)^{-1} p(x,\psi) \right) \Bigg).
\end{align*}
Since $p(x,\mu)=p(x,\mu_x)$ for all $(x,\mu)\in\BX\times\bN^*$ and
$x_1<\cdots<x_n <y$, we obtain that the above equals
\begin{align*}
  &\sum_{\psi \le \mu}\Bigg(\alpha(y)^{-1} p(y,\psi+\delta_y) 
\prod_{x \in \psi} \alpha(x)^{-1} p(x,\psi) \prod_{x\in \mu-\psi} (1-\alpha(x)^{-1} p(x,\psi))\\
  &\quad\quad \quad +(1-\alpha(y)^{-1} p(y,\psi))\prod_{x \in \psi} \alpha(x)^{-1} 
p(x,\psi) \prod_{x \in \mu-\psi} (1-\alpha(x)^{-1} p(x,\psi)) \Bigg).
\end{align*}
Since $p(y,\psi+\delta_y)=p(y,\psi)$, this equals
\begin{align*}
  \sum_{\psi \le \mu}  \prod_{x \in \psi} \alpha(x)^{-1} p(x,\psi) 
\prod_{x\in \mu-\psi} (1-\alpha(x)^{-1} p(x,\psi))
= K_{\kappa,\alpha}(\mu,\mathbf{N}^*).
\end{align*}
This finishes the induction.\qed

\bigskip

The following thinning representation is the main result
of this section. It extends \cite[Proposition 4.1]{HTHou17}
to general Borel spaces.

\begin{theorem}\label{t2.1} Assume that $\lambda$ is finite and diffuse
and that (Dom2) holds. Then
\begin{align}
\BQ_{\kappa,\alpha}:=\iint \I\{\psi\in\cdot\}\,K_{\kappa,\alpha}(\mu,\mathrm{d}\psi)\,\Pi_{\alpha\lambda}(\mathrm{d}\mu)
\end{align}
is the distribution of a Gibbs process with PI $\kappa$.
\end{theorem}
{\em Proof:} The proof will use Corollary \ref{c2.4}, to be proved later.
Let $f\colon\bN^*\to\R_+$ be measurable. Then
\begin{align*}
\iint f(\psi)\,\BQ_{\kappa,\alpha}(\mathrm{d}\psi)
=\sum^\infty_{n=0}\iint\I\{\psi(\BX)=n\}f(\psi)\,K_{\kappa,\alpha}(\mu,\mathrm{d}\psi)\,\Pi_{\alpha\lambda}(\mathrm{d}\mu).
\end{align*}
Since $\alpha\lambda$ is diffuse, a Poisson process with this distribution is
simple (\cite[Proposition 6.9]{LastPenrose17}), so that the above equals
\begin{align*}
\sum^\infty_{n=0}&\iint f(\delta_{x_1}+\cdots+\delta_{x_n})
\I\{x_1<\cdots<x_n\}\prod^n_{i=1}\alpha(x_i)^{-1}p(x_i,\delta_{x_1}+\cdots+\delta_{x_n})\\
&\times\,\prod_{x\in\mu-(\delta_{x_1}+\cdots+\delta_{x_n})}
\big(1-\alpha(x_i)^{-1}p(x,\delta_{x_1}+\cdots+\delta_{x_n})\big)
\,\mu^{(n)}(\mathrm{d}(x_1,\ldots,x_n))\,\Pi_{\alpha\lambda}(\mathrm{d}\mu),
\end{align*}
where, for $n=0$, we interpret $\delta_{x_1}+\cdots+\delta_{x_n}$ as the zero measure.
By the multivariate Mecke equation (\cite[Theorem 4.4]{LastPenrose17}), this equals
\begin{align*}
\sum^\infty_{n=0}&\int f(\delta_{x_1}+\cdots+\delta_{x_n})
\prod^n_{i=1}p(x_i,\delta_{x_1}+\cdots+\delta_{x_n})P(x_1,\ldots,x_n)
\,\lambda_n(\mathrm{d}(x_1,\ldots,x_n)),
\end{align*}
where $\lambda_n$ is the restriction of $\lambda^n$ to $\{x_1<\cdots<x_n\}$ and
\begin{align*}
P(x_1,\ldots,x_n):=\int \prod_{x\in\mu}(1-\alpha(x)^{-1}p(x,\delta_{x_1}+\cdots+\delta_{x_n}))
\,\Pi_{\alpha\lambda}(\mathrm{d}\mu).
\end{align*}
By \cite[Exercise 3.6]{LastPenrose17} we have that
\begin{align*}
P(x_1,\ldots,x_n)
=\exp\Big[-\int p(x,\delta_{x_1}+\cdots+\delta_{x_n})\,\lambda(\mathrm{d}x)\Big],
\end{align*}
so that
\begin{align*}
\int f(\psi)&\,\Pi_{\kappa,\lambda}(\mathrm{d}\psi)=\sum^\infty_{n=0}\int f(\delta_{x_1}+\cdots+\delta_{x_n})
\prod^n_{i=1}p(x_i,\delta_{x_1}+\cdots+\delta_{x_n})\\
&\times\exp\Big[-\int p(x,\delta_{x_1}+\cdots+\delta_{x_n})\,\lambda(\mathrm{d}x)\Big]
\,\lambda_n(\mathrm{d}(x_1,\ldots,x_n)).
\end{align*}
By assumption (Dom2) and Remark \ref{rbounded} all elements
of $\bN^*$ are stable.
Inserting above the result of Corollary \ref{c2.4} and
the definition \eqref{ethprob}, we obtain that
\begin{align*}
\int f(\psi)\,\BQ_{\kappa,\alpha}(\mathrm{d}\psi)&=
Z_{\BX}(0)^{-1}e^{-\lambda(\BX)}
\sum^\infty_{n=0}\int f(\delta_{x_1}+\cdots+\delta_{x_n})\kappa_n(x_1,\ldots,x_n,0)\\
&\quad\times \I\{x_1<\cdots<x_n\}\,\lambda^n(\mathrm{d}(x_1,\ldots,x_n)),
\end{align*}
where $\kappa_0\equiv 1$. Since $\lambda$ is diffuse
and $\kappa_n(\cdot,0)$ is symmetric $\lambda^n$-a.e., we obtain that
\begin{align*}
\int f(\psi)\,\BQ_{\kappa,\lambda}(\mathrm{d}\psi)
=Z_{\BX}(0)^{-1}\int f(\mu)e^{-H(\mu,0)}\,\Pi_{\lambda}(\mathrm{d}\mu).
\end{align*}
Therefore the assertion follows from Corollary \ref{c2.3}.\qed

\bigskip

The following lemma provides useful information on the function $p$.
It does not require assumption (Dom2).
Note that $\psi_{-\infty}=0$ and $Z_\emptyset(\psi)=1$
for each $\psi\in\bN^*$. We set $(\infty,\infty):=\emptyset$.

\begin{lemma}\label{l2.2} Suppose that $\psi\in\bN^*$ is stable 
and let $z,w\in\BX\cup\{-\infty,\infty\}$ such that $z<w$. Then
\begin{align}\label{e2.2}
\int \I\{x\in (z,w]\}p(x,\psi)\,\lambda(\mathrm{d}x)
&=\lambda(z,w]+\log Z_{(z,\infty)}(\psi_{(-\infty,z]})
-\log Z_{(w,\infty)}(\psi_{(-\infty,w)})\\\notag
  &\quad+\int\big(\log Z_{(x,\infty)}(\psi_x+\delta_x)-\log Z_{(x,\infty)}(\psi_x)\big)
\,\psi_{(z,w)}(\mathrm{d}x).
\end{align}
\end{lemma}
{\em Proof:} For $x\in\BX$ and $\psi\in\bN^*$ we abbreviate
$Z_x(\psi):=Z_{(x,\infty)}(\psi)$. 
Since $\lambda$ is diffuse, we have $\Pi_\lambda(\{\mu\in\bN:\mu(\{x\})\ge 1\})=0$.
Therefore we obtain from definition \eqref{epartition}
 that $Z_x(\psi)=Z_{[x,\infty)}(\psi)$. Since $\psi$ is stable, this is a
finite number. In the following we write
$\lambda(x,\infty):=\lambda((x,\infty))$ and similar for other intervals.
By definition,
\begin{align*}
Z_x(\psi_x)&=e^{-\lambda(x,\infty)}
+\sum^\infty_{n=1}\int e^{-H(\mu_{(x,\infty)},\psi_x)}\I\{\mu(x,\infty)=n\}\,\Pi_\lambda(\mathrm{d}\mu)\\
&=e^{-\lambda(x,\infty)}
+\sum^\infty_{n=1}\iint e^{-H(\mu_{[y,\infty)},\psi_x)}
\I\{y\in(x,\infty)\}\\
&\quad\times\I\{\mu(x,y)=0,\mu(y,\infty)=n-1\}\,\mu(\mathrm{d}y)\,\Pi_\lambda(\mathrm{d}\mu),
\end{align*}
where we have used that $\Pi_\lambda$ is the distribution of a simple
point process. By the Mecke equation this equals
\begin{align*}
e^{-\lambda(x,\infty)}
+\sum^\infty_{n=1}&\iint e^{-H(\mu_{(y,\infty)}+\delta_y,\psi_x)}
\I\{y\in(x,\infty)\}\\
&\quad\times\I\{\mu(x,y)=0,\mu(y,\infty)=n-1\}\,\Pi_\lambda(\mathrm{d}\mu)\,\lambda(\mathrm{d}y),
\end{align*}
where we have used that a Poisson process with a diffuse
intensity measure does not have fixed atoms.
Using the complete independence of a Poisson process and then
the formula for empty space probablities, we obtain that
\begin{align*}
Z_x(\psi_x)=e^{-\lambda(x,\infty)}
+\iint \I\{y\in(x,\infty)\}e^{-\lambda(x,y)} e^{-H(\mu_{(y,\infty)}+\delta_y,\psi_x)}
\,\Pi_\lambda(\mathrm{d}\mu)\,\lambda(\mathrm{d}y).
\end{align*}
By the definition of the Hamiltonian we have that
\begin{align*}
e^{-H(\mu_{(y,\infty)}+\delta_y,\psi_x)}=\kappa(y,\psi_x)e^{-H(\mu_{(y,\infty)},\psi_x+\delta_y)},
\end{align*}
so that
\begin{align}\label{e2.7}
Z_x(\psi_x)=e^{-\lambda(x,\infty)}
+e^{\lambda(x)}\int \I\{y\in(x,\infty)\}
e^{-\lambda(y)}\kappa(y,\psi_x) Z_y(\psi_x+\delta_y)\,\lambda(\mathrm{d}y),
\end{align}
where $\lambda(y):=\lambda(-\infty,y]$.

We now argue that we can assume that $\BX=\R$.
Consider the Borel isomorphism
$\varphi\colon \BX\to U:=\varphi(\BX)$, interpreted as a mapping
from $\BX$ to $\R$.
We define a diffuse measure $\lambda'$ on $\R$
by $\lambda':=\lambda\circ\varphi^{-1}$ (the image of $\lambda$ under
$\varphi$). Recall that $\bN^*(\R)$ denotes the set of all finite and simple counting
measures on $\R$ and define
a measurable mapping $\kappa'\colon \R\times \bN^*(\R)\to\R_+$
by $\kappa'(u,\mu):=\I\{u\in U\}\kappa(\varphi^{-1}(u),\mu_U\circ\varphi)$.
(If $\mu_U=\delta_{s_1}+\cdots+\delta_{s_n}$ then
$\mu_U\circ\varphi=\delta_{\varphi^{-1}(s_1)}+\cdots+\delta_{\varphi^{-1}(s_n)})$.)
Then $\kappa'$ satisfies the cocycle assumption \eqref{ecocycle}
and we denote the associated partition functions by
$Z'_B$ for $B\in\mathcal{B}(\R)$ (the system of Borel subsets of $\R$).
It is easy to show that
\begin{align}\label{e2.57}
Z'_B(\mu)=Z_{\varphi^{-1}(B)}(\mu_U\circ\varphi),\quad \mu\in \bN^*(\R),\,B\in \mathcal{B}(\R).
\end{align}
Let  $\mu:=\psi\circ \varphi^{-1}$ and suppose that \eqref{e2.2}
holds with $(\BX,\lambda,\kappa,\psi)$ replaced with
$(\R,\lambda',\kappa',\mu)$. This means
that for all $s,t\in\R\cup\{-\infty,\infty\}$
\begin{align*}
\int \I\{u\in (s,t]\}p'(u,\mu)\,\lambda'(\mathrm{d}u)
&=\lambda'(s,t]+\log Z'_{(s,\infty)}(\mu_{(-\infty,s]})
-\log Z'_{(t,\infty)}(\mu_{(-\infty,s]})\\\notag
&\quad+\int\big(\log Z'_{(u,\infty)}(\mu_u+\delta_u)-\log Z'_{(u,\infty)}(\mu_u)\big)\,\mu_{(s,t)}(\mathrm{d}u),
\end{align*}
where $p'$ is defined in terms of $\kappa'$ as $p$ in terms of $\kappa$.
Applying this formula with $(s,t):=(\varphi(z),\varphi(w))$ 
(where $(\varphi(-\infty),\varphi(\infty)):=(-\infty,\infty)$) and using
\eqref{e2.57}, yields \eqref{e2.2}.

From now on we will assume that $\BX=\R$. 
$a\in\R\cup\{-\infty,\infty\}$ we set $\psi^+_a:=\psi_{[a,\infty)}$.
Let $a,b\in\R\cup\{-\infty,\infty\}$ such that
$\psi(a,b)=0$. Then 
\begin{align*}
\psi_x=\psi^+_a,\quad x\in(a,b].
\end{align*}
Since the left-hand side of \eqref{e2.7} is finite,
so is the right-hand side.
This shows that the function $x\mapsto f(x):=Z_x(\psi_x)$
from $(a,b]$ to $\R$ is continuous and of totally bounded
variation. Now we use Lebesgue-Stieltjes calculus;
see e.g.\ \cite[Appendix A4]{LastBrandt95}.
By the product rule and $d_xe^{-\lambda(x,\infty)}=-e^{-\lambda(x,\infty)}\lambda(\mathrm{d}x)$
we have that
\begin{align*}
df(x)&=e^{-\lambda(x,\infty)}\lambda(\mathrm{d}x)
+e^{\lambda(x)}\bigg[\int^{\infty}_x 
e^{-\lambda(y)}\kappa(y,\psi^+_a) Z_y(\psi^+_a+\delta_y)\,\lambda(\mathrm{d}y)\bigg]\,\lambda(\mathrm{d}x)\\
&\quad-\kappa(x,\psi_a) Z_x(\psi^+_a+\delta_x)\,\lambda(\mathrm{d}x),
\end{align*}
where we note that $\kappa(x,\psi_a) Z_a(\psi^+_a+\delta_x)<\infty$
for $\lambda$-a.e.\ $x\in[a,b]$. Therefore,
\begin{align*}
df(x)&=e^{-\lambda(x,\infty)}\lambda(\mathrm{d}x)+\big(f(x)-e^{-\lambda(x,\infty)}\big)\,\lambda(\mathrm{d}x)
-\kappa(x,\psi_x) Z_x(\psi^+_a+\delta_x)\,\lambda(\mathrm{d}x)\\
&=f(x)\,\lambda(\mathrm{d}x)-\kappa(x,\psi_a) Z_x(\psi^+_a+\delta_x)\,\lambda(\mathrm{d}x),
\end{align*}
that is
\begin{align}\label{ediif}
d Z_x(\psi_x)=Z_x(\psi_x)\,\lambda(\mathrm{d}x)
-\kappa(x,\psi_x) Z_x(\psi_x+\delta_x)\,\lambda(\mathrm{d}x),\quad \text{on $(a,b]$}.
\end{align}
Therefore we obtain from the definition \eqref{ethprob} that
\begin{align}\label{e2.13}\notag
\int^b_{a} p(x,\psi)\,\lambda(\mathrm{d}x)&=\lambda(a,b]-\int^b_{a} Z_x(\psi_x)^{-1}\,d Z_x(\psi_x)\\
&=\lambda(a,b]+\log Z_a(\psi^+_a)-\log Z_b(\psi^+_a),
\end{align}
where we have used \cite[Corollary A4.11]{LastBrandt95}.
Note that
\begin{align*}
\lim_{z\to -\infty}Z_z(\psi^+_{z})=Z_\BX(0),\qquad 
\lim_{z\to \infty} Z_z(\psi^+_z)=1.
\end{align*}

There exist $n\in\N_0$ and $x_1,\ldots,x_n\in\BX$
such that $x_1<\cdots <x_n$ and $\psi=\delta_{x_1}+\cdots+\delta_{x_n}$.
If $n=0$, then the assertion \eqref{e2.2} follows from \eqref{e2.13}. 
Hence we can assume that $n\ge 1$. 
By \eqref{e2.13} we have for each $i\in\{0,\ldots,n\}$,
\begin{align}\label{e2.14}
\int^{x_{i+1}}_{x_i} p(x,\psi)\,\lambda(\mathrm{d}x)
=\lambda(x_i,x_{i+1}]+\log Z_{x_i}(\psi_{(-\infty,x_i]})
-\log Z_{x_{i+1}}(\psi_{(-\infty,x_i]}), 
\end{align}
where $x_0:=-\infty$ and $x_{n+1}:=\infty$.

Let $z,w\in\R$ with $z<w$. Then there exist $j,k\in\{0,\ldots,n\}$
with $j\le k$ such that $z\in [x_j,x_{j+1})$ and $w\in [x_k,x_{k+1})$.
For $i\in\{0,\ldots,n\}$ we set $\psi^i:=\delta_{x_1}+\cdots+\delta_{x_i}=\psi_{(-\infty,x_i]}$.
Let us first assume that $x_k<w$. 
From \eqref{e2.13} and \eqref{e2.14} we then derive that 
\begin{align*}
\int^w_z &p(x,\psi)\,\lambda(\mathrm{d}x)
=\int^{x_{j+1}}_zp(x,\psi)\,\lambda(\mathrm{d}x)
+\sum^{k-1}_{i=j+1}\int^{x_{i+1}}_{x_i}p(x,\psi)\,\lambda(\mathrm{d}x)
+\int^w_{x_k}p(x,\psi)\,\lambda(\mathrm{d}x)\\
&=\lambda(z,w]+\log Z_{z}(\psi^j)-\log Z_{x_{j+1}}(\psi^{j})\\
&\quad+\sum^{k-1}_{i=j+1}\log Z_{x_i}(\psi^{i})
-\sum^{k-1}_{i=j+1}\log Z_{x_{i+1}}(\psi^{i})+\log Z_{x_{k}}(\psi^{k})
-\log Z_{w}(\psi^{k}).
\end{align*}
Reordering terms yields,
\begin{align*}
\int^w_z &p(x,\psi)\,\lambda(\mathrm{d}x)\\
&=\lambda(z,w]+\log Z_z(\psi^+_z)
-\log Z_w(\psi_w)
+\sum^{k}_{i=j+1}\log Z_{x_i}(\psi^{i})
-\sum^{k-1}_{i=j}\log Z_{x_{i+1}}(\psi^{i})\\
&=\lambda(z,w]+\log Z_z(\psi^+_z)-\log Z_w(\psi_w)
+\sum^{k}_{i=j+1}\big(\log Z_{x_i}(\psi^{i})-\log Z_{x_i}(\psi^{i-1})\big).
\end{align*}
This is equivalent to \eqref{e2.2}. If $w=x_k$ we have that
\begin{align*}
  \int^w_z &p(x,\psi)\,\lambda(\mathrm{d}x)
=\lambda(z,w]+\log Z_{z}(\psi^j)
+\sum^{k-1}_{i=j+1}\log Z_{x_i}(\psi^{i})
-\sum^{k-1}_{i=j}\log Z_{x_{i+1}}(\psi^{i})\\
&=\lambda(z,w]+\log Z_z(\psi^j)-\log Z_{x_k}(\psi^{k-1})
+\sum^{k-1}_{i=j+1}\big(\log Z_{x_i}(\psi^{i})-\log Z_{x_{i+1}}(\psi^{i})\big),
\end{align*}
which is again equivalent to \eqref{e2.2}.
\qed

\bigskip

The following corollary has been crucial in the proof of Theorem \ref{t2.1}.

\begin{corollary}\label{c2.4} 
Suppose that $\psi\in\bN^*$ is stable. Then
\begin{align}\label{exp123}
\exp\bigg[-\int p(x,\psi)\,\lambda(\mathrm{d}x)\bigg]
=e^{-\lambda(\BX)}Z_{\BX}(0)^{-1}
\prod_{y\in\psi}\frac{Z_{(y,\infty)}(\psi_y)}{Z_{(y,\infty)}(\psi_y+\delta_y)}.
\end{align}
In particular $\int p(x,\psi)\,\lambda(\mathrm{d}x)<\infty$.
\end{corollary}
{\em Proof:} Taking in Lemma \ref{l2.2} $z=-\infty$ and $w=\infty$
yields the first assertion.
To prove the second, we take $y\in\psi$. Then
$Z_{(y,\infty)}(\psi_y+\delta_y)=Z_{(y,\infty)}(\psi_{(-\infty,y]})$.
Since $\psi$ is stable, this is finite.
Therefore the right-hand side of \eqref{exp123} does not vanish and the result
follows. \qed

\section{Poisson embedding of finite Gibbs processes}\label{secembedding}

As in Section \ref{secthinning} we consider a diffuse 
finite measure $\lambda$ on $\BX$ and a measurable function
$\kappa\colon\BX\times\bN\to\R_+$, satisfying the cocycle assumption
\eqref{ecocycle} for all $(x,y,\mu)\in\BX^2\times\bN$.
In this section we construct a (finite) Gibbs process by a recursively 
defined embedding 
into a Poisson process on $\BX\times\R_+$ with intensity
measure $\lambda\otimes\lambda_1$, where $\lambda_1$ denotes
Lebesgue measure on $\R_+$. For (marked) point processes 
on $\R_+$ this embedding technique is well-known;
see \cite{BreMass96}. To the best of our knowledge it has
never been used in a spatial setting.
As in Theorem \ref{t2.1} we use the function $p$ defined by
\eqref{ethprob}. However, we do not need the local stability
assumption (Dom2).
Even if this assumption holds, we find it more convenient
to work with embedding rather than with a probabilistic thinning
version of Theorem \ref{t2.1}.

Let $\bN^*_\ell(\BX\times\R_+)$ denote the space of all
simple counting measures $\psi$ on $\BX\times\R_+$ such
that $\psi(B)<\infty$ for each measurable $B\subset\BX\times \R_+$ 
with $(\lambda\otimes\lambda_1)(B)<\infty$. (Again we equip this space
with the standard $\sigma$-field.) Recall that
$\bN^*$ denotes the space of simple and finite counting measures
on $\BX$ and note that the elements of
$\bN^*_\ell(\BX\times\R_+)$ are not assumed to be finite.
Using the total ordering on $\BX$
we define for each $n\in\N$
a mapping $x_n(\cdot)\colon \bN^*_\ell(\BX\times\R_+)\to \BX\cup\{-\infty,\infty\}$
as follows. Let $\psi\in \bN^*_\ell(\BX\times\R_+)$. If $\int p(x,0)\,\lambda(\mathrm{d}x)=\infty$, then we set
$x_1(\psi):=-\infty$.
If $\int p(x,0)\,\lambda(\mathrm{d}x)<\infty$ then 
\begin{align*}
\psi(\{(x,t)\in \BX\times\R_+:t\le p(x,0)\})<\infty 
\end{align*}
and we set
\begin{align}\label{ex1psi}
x_1(\psi):=\min\{x\in\BX:\text{there exists $t\ge 0$ such that $(x,t)\in\psi$ and $t\le p(x,0)$}\},
\end{align}
where $\min\emptyset:=\infty$.
Inductively we define $x_n(\psi)$ for all $n\in\N$. If 
$x_n(\psi)\notin\BX$ then we set $x_{n+1}(\psi):=x_n(\psi)$.
If $x_n(\psi)\in\BX$ and 
\begin{align*}
\int \I\{x_n(\psi)<x\}p(x,\delta_{x_1(\psi)}+\cdots+\delta_{x_n(\psi)})\,\lambda(\mathrm{d}x)=\infty,
\end{align*}
then we set $x_{n+1}(\psi):=-\infty$. Otherwise we define
\begin{align}\label{exnpsi}
&x_{n+1}(\psi)\\ \notag
&:=\min\{x>x_n(\psi):\text{there exists $t\ge 0$ s.t.\ $(x,t)\in\psi$ and 
$t\le p(x,\delta_{x_1(\psi)}+\cdots+\delta_{x_n(\psi)})$}\}.
\end{align}
Define the {\em embedding operator} $T\colon \bN^*_\ell(\BX\times\R_+)\to\bN^*(\BX)$ by
\begin{align}\label{eTpsi}
T(\psi):=\I\{\tau(\psi)<\infty\}
\sum^{\tau(\psi)}_{n=1}\delta_{x_n(\psi)},
\end{align}
where $\tau(\psi):=\sup\{n\ge 1: x_n(\psi)\in\BX\}$.
Equation \eqref{eTmeas} below provides an alternative
representation of $T$, which shows that $T$ is measurable.

\begin{theorem}\label{tembedding} Assume that $\lambda$ is finite
and diffuse and that $\Pi_\lambda$-a.e.\ $\psi\in\bN^*$ are stable.
Assume that $\Phi$ is a Poisson process on $\BX\times\R_+$ with
intensity measure $\lambda\otimes\lambda_1$. Then $T(\Phi)$ is a Gibbs process
with PI $\kappa$.
\end{theorem}
{\em Proof:} Given $n\in\N_0$, $x_1,\ldots,x_n\in\BX$ and
$i\in\{0,\ldots,n\}$ we define
\begin{align*}
A_i(x_1,\ldots,x_n):=\{(x,t)\in\BX\times\R_+:x_i<x<x_{i+1}, t\le p(x,\delta_{x_1}+\cdots+\delta_{x_i})\},
\end{align*}
where we set $x_0:=-\infty$ and $x_{n+1}:=\infty$ with the convention
that $-\infty<y<\infty$ for each $y\in\BX$. 
For $n\in\N_0$ we set $B_n(x_1,\ldots,x_n):=\cup^n_{i=0}A_i(x_1,\ldots,x_n)$.
Let $H:=\{\psi\in \bN^*(\BX):\int p(x,\psi)\,\lambda(\mathrm{d}x)<\infty\}$ and
define
\begin{align*}
C_n:=\{(x_1,\ldots,x_n)\in\BX^n:\delta_{x_1}\in H,\delta_{x_1}+\delta_{x_2}\in H,
\ldots,\delta_{x_1}+\cdots+\delta_{x_n}\in H\},
\end{align*}
where $C_0:=\I\{0\in H\}$.
By definition of $T$ we have for all $B\in\cX$
\begin{align}\label{eTmeas}
&T(\psi)(B)=\sum^\infty_{n=1}\sum^n_{i=1}\int
\I\{t_1\le p(x_1,0),\ldots,t_n\le p(x_n,\delta_{x_1}+\cdots+\delta_{x_{n-1}})\}
\I\{x_1<\cdots<x_n\}\\ \notag
&\times\I\{(x_1,\ldots,x_{n-1})\in C_{n-1},\psi(B_{n}(x_1,\ldots,x_n))=0\}
\I\{x_i\in B\}\,\psi^{(n)}(\mathrm{d}(x_1,t_1,\ldots,x_n,t_n)).
\end{align}
This shows that $T$ is measurable.

We next prove that
\begin{align}\label{e4.29}
\BP(x_n(\Phi)=-\infty)=0,\quad n\in\N.
\end{align}
First we note that $x_1(\Phi)=-\infty$ iff
$\int p(x,0)\,\lambda(\mathrm{d}x)=\infty$. Since
$\Pi_\lambda(\{0\})>0$, the empty measure $0$ is stable,
so that Corollary \ref{c2.4} shows that this case cannot occur.
For $n \in \N$ we have that
\begin{align*}
\BP(&x_{n+1}(\Phi)=-\infty) \le \sum\limits_{k=1}^{n}\BP(T(\Phi)(\BX)=k,x_{k+1}(\Phi)=-\infty).
\end{align*}
Hence, it suffices to show that $\BP(T(\Phi)(\BX)=n,x_{n+1}(\Phi)=-\infty)=0$ for all $n \in \N$. We have
\begin{align*}
\BP(&x_{n+1}(\Phi)=-\infty)\\
&=\BE\bigg[\int
\I\{t_1\le p(x_1,0),\ldots,t_n\le p(x_n,\delta_{x_1}+\cdots+\delta_{x_{n-1}})\}
\I\{x_1<\cdots<x_n\} \\
&\qquad\times\I\{\Phi(B_{n-1}(x_1,\ldots,x_n))=0,(x_1,\ldots,x_n)\notin C_{n}\}
\,\Phi^{(n)}(\mathrm{d}(x_1,t_1,\ldots,x_n,t_n))\bigg].
\end{align*}
By the multivariate Mecke equation this equals
\begin{align*}
\BE\bigg[&\int
p(x_1,0)\cdots p(x_n,\delta_{x_1}+\cdots+\delta_{x_{n-1}})
\I\{x_1<\cdots<x_n\} \\
&\qquad\times\I\{\Phi(B_{n-1}(x_1,\ldots,x_{n-1}))=0,(x_1,\ldots,x_{n})\notin C_{n}\}
\,\lambda^n(\mathrm{d}(x_1,\ldots,x_n))\bigg]\\
=&\int
p(x_1,0)\cdots p(x_n,\delta_{x_1}+\cdots+\delta_{x_{n-1}})
\I\{x_1<\cdots<x_n\} \\
&\quad\times \BP(\Phi(B_{n-1}(x_1,\ldots,x_{n-1}))=0)\I\{(x_1,\ldots,x_n)\notin C_{n}\}
\,\lambda^n(\mathrm{d}(x_1,\ldots,x_n)).
\end{align*}
Our assumptions allow to apply Corollary \ref{c2.4}, so that
\begin{align}\label{e4.44}
\int \I\{x_n<x\}p(x,\delta_{x_1}+\cdots+\delta_{x_{n-1}})\,\lambda(\mathrm{d}x)<\infty
\end{align}
for $\lambda^n$-a.e.\ $(x_1,\ldots,x_n)$.
Hence \eqref{e4.29} follows.

Now we take a measurable $f\colon \bN^*(\BX)\to\R_+$.
Taking into account \eqref{e4.29}, we obtain
for each $n\in\N_0$ similarly as above that
\begin{align*}
\BE [f(T(\Phi))]=\sum^\infty_{n=0}
\int &f(\delta_{x_1}+\cdots+\delta_{x_n})
p(x_1,0)\cdots  p(x_n,\delta_{x_1}+\cdots+\delta_{x_{n-1}})\I\{x_1<\cdots<x_n\}\\
&\times \BP(\Phi(B_n(x_1,\ldots,x_n))=0)\,\lambda^n(\mathrm{d}(x_1,\ldots,x_n)),
\end{align*}
with an obvious interpretation of the summand for $n=0$.
Whenever $x_1<\cdots <x_n$ we have that
\begin{align*}
\BP(\Phi(B_n(x_1,\ldots,x_n))=0)&=\prod^n_{i=0} \BP(\Phi(A_i(x_1,\ldots,x_n))=0)\\
&=\prod^n_{i=0}\exp\bigg[-\int\I\{x_i<x<x_{i+1}\}p(x,\delta_{x_1}+\cdots+\delta_{x_i})\,\lambda(\mathrm{d}x)\bigg].
\end{align*}
Using here Lemma \ref{l2.2}, we can conclude the assertion as in the
proof of Theorem \ref{t2.1}.\qed

\bigskip

\begin{remark}\label{r22}\rm Suppose that $\kappa$ satisfies
assumption (Dom2). Then we can replace $\Phi$ by its restriction
to $\{(x,t)\in\BX\times \R_+:t\le \alpha(x)\}$.
\end{remark}

\bigskip
 Later we need the following useful consequence of Theorem \ref{tembedding}.
Here we do not assume that the measure $\lambda$ is finite.
For technical reasons we assume that $\BX$ is a complete separable metric
space and that each set in $\cX_0$ is bounded.

\begin{lemma}\label{lthreecopling} Let $\xi,\xi'$ be two Gibbs processes
on $\BX$ with Papangelou intensities $\kappa,\kappa'$ which both satisfy
(Dom2) with the same function $\alpha$. 
Then there exists a Poisson process $\eta$ with intensity measure 
$\alpha\lambda$ and point processes $\tilde\xi,\tilde{\xi}'$, 
such that $\xi\overset{d}{=}\tilde\xi$, $\xi'\overset{d}{=}\tilde\xi'$
and $\tilde{\xi}\le \eta$ and $\tilde{\xi}'\le \eta$ almost surely.
\end{lemma}
{\em Proof:} Let $B\in\cX_0$. The restriction $\xi_B$ of $\xi$
to $B$ is a Gibbs process with PI
\begin{align*}
\kappa^B(x,\mu):=\int \kappa(x,\mu+\psi)\,\BP(\xi_{\BX\setminus B}\in\mathrm{d}\psi\mid \xi_B=\mu),
\quad (x,\mu)\in B\times\bN_B.
\end{align*}

Note that $\kappa^B\le \alpha$. A similar assertion applies to
$\xi'$ and its PI $\kappa'^B$.
Let $\Phi$ be a Poisson process on $\BX\times\R_+$ with
intensity measure $\lambda\otimes\lambda_+$. Since $\lambda$ is diffuse
we can assume that $\Phi$ is simple, that is a random element
of $\bN^*_{\ell}(\BX\times\R_+)$. Let $T^B$ and $T'^B$ be the
embedding operators associated with $\kappa^B$ resp.\
$\kappa'^B$. Define 
$(\chi^B,\chi'^B):=(T^B(\Phi_{B\times\R_+}),T'^B(\Phi_{B\times\R_+}))$.
By Theorem \ref{tembedding} we have that
$\chi^B\overset{d}{=}\xi_B$ and $\chi'^B\overset{d}{=}\xi'_B$.
By definition of the embedding and Lemma \ref{l4.1} we have
that $\chi^B$ and $\chi'^B$ are both smaller than 
\begin{align*}
\Phi^B:=\int \I\{x\in\cdot,x\in B, t\le \alpha(x)\}\,\Phi(\mathrm{d}(x,t)),
\end{align*} 
which is a Poisson process with intensity measure $\alpha\lambda_B$.
Now we argue exactly as in the proof of \cite[Corollary 3.4]{GeorYoo05}.
We take a sequence $(B_n)_n$ of bounded and closed sets
such that $B_n\uparrow\BX$. Then 
$\chi^{B_n}\overset{d}{\rightarrow}\xi$,
$\chi'^{B_n}\overset{d}{\rightarrow}\xi'$
and $\Phi^{B_n}\overset{d}{\rightarrow}\eta$,
where $\eta$ is a Poisson process with intensity measure
$\alpha\lambda$ and where we refer to \cite[Chapter 4]{Kallenberg17}
for the theory of weak convergence of point process distributions.
It follows from \cite[Theorem 16.3]{Kallenberg}
that the above sequences are all tight. A standard argument
shows that $(\chi^{B_n},\chi'^{B_n},\Phi^{B_n})$, $n\in\N$, 
is tight and hence converges in distribution to 
$(\tilde\xi,\tilde\xi', \eta)$ along some subsequence.
Then $\xi\overset{d}{=}\tilde\xi$ and $\xi'\overset{d}{=}\tilde\xi'$.
Moreover, $\eta$ is a Poisson process with
intensity measure $\alpha\lambda$. We can assume that
$(\tilde\xi,\tilde\xi', \eta)$ is defined on the original
probability space $(\Omega,\cF,\BP)$.
Since the set $\{(\mu,\mu',\psi)\in\bN^3:\mu\le \psi,\mu'\le \psi\}$
is closed, it follows from the Portmanteau theorem that 
$\BP(\tilde\xi\le \eta,\tilde\xi'\le \eta)=1$.
\qed

\section{Disagreement coupling}\label{secdisagreement}

In this section we return to the general setting of Section \ref{secGibbs},
that is we consider a  $\sigma$-finite measure $\lambda$
along with a measurable 
$\kappa\colon\BX\times\bN\to\R_+$ satisfying
\eqref{ecocycle} for all $(x,y,\mu)\in\BX^2\times\bN$.
We assume that $\lambda$ is diffuse.
For $\psi\in\bN$ and $B\in\cX_0$ we recall the definition \eqref{ekappapsi}
of $\kappa_\psi$ and that  $\kappa_{B,\psi}$ is 
the restriction of $\kappa_\psi$ to $B\times \bN_B$.
Using $\kappa_{B,\psi}$ instead of $\kappa$ we can define
the function $p_{B,\psi}\colon W\times\bN^*(B)\to[0,1]$ by
\eqref{ethprob} and the function 
$T_{B,\psi}\colon \bN^*_{\ell}(B\times\R_+)\to \bN^*(B)$
by \eqref{eTpsi}. Note that this mapping depends on the
chosen ordering on $B$.
Let  $\Phi$ be a Poisson process
on $\BX\times\R_+$  with intensity measure $\lambda\otimes\lambda_1$.
We will use Theorem \ref{tembedding} by applying the mapping
$T_{Z,\Psi}$ to the restriction $\Phi_{Z\times \R_+}$
for suitable random sets $Z$ and (recursively defined) point processes
$\Psi$. 

In the following we fix $W\in\cX_0$.
We say that $\psi\in \mathbf{N}_{W^c}$ is a {\em regular boundary condition} ({\em for $W$})
if 
\begin{align}\label{e62.3}
Z_B(\mu_{W\setminus B}+\psi)<\infty,\quad B\subset W,\,B\in\cX,\, \Pi_{\lambda}\text{-a.e.\ $\mu$}.
\end{align}
This means that there exists a measurable $H\subset\bN$ such that
$\Pi_\lambda(H)=1$ and such that $\mu_{W\setminus B}+\psi$ is for
all $\mu\in H$ stable on $W$; cf.\ (Dom1).

\begin{remark}\label{rbounded0}\rm If the PI $\kappa$ 
satisfies (Dom2), then each $\psi\in\bN_{W^c}$ is a regular
boundary condition for $W\in\cX_0$; see Remark \ref{rbounded}.
\end{remark}

\begin{remark}\label{rbounded2}\rm Let $\xi$ be a Gibbs process with PI $\kappa$ and let $W,B\in\cX_0$
with $B\subset W$. From \eqref{e2.3} and the DLR-equations
we obtain that
\begin{align*}
\int \I\{Z_B(\mu_{W\setminus B}+\psi)<\infty\}\,\Pi_{W,\psi}(\mathrm{d}\mu)=1,
\quad \BP(\xi_{B^c}\in\cdot)\text{-a.e.\ $\psi$},
\end{align*}
where $\Pi_{W,\psi}$ is the conditional distribution of $\xi_W$ given 
$\xi_{W^c}=\psi$. Therefore \eqref{e62.3} seems to be a reasonable assumption,
at least if the exceptional set is allowed to depend
on $B\subset W$. If (Dom1) holds for all $\psi\in\bN$, then each boundary condition
is regular for $W$.
\end{remark}

We construct a special coupling of
two Gibbs processes on $W\in\cX_0$ with Papangelou intensities 
$\kappa_{W,\psi}$ and $\kappa_{W,\psi'}$ for regular boundary conditions
$\psi,\psi'\in\bN_{W^c}$. 
This extends the coupling construction in \cite{HTHou17} (see also
\cite[Theorem 1]{BHLV20}) to general Borel spaces.
While the coupling in \cite{HTHou17} is
based on the thinning construction from Theorem \ref{t2.1},
we find it more convenient  to work with the
more explicit Poisson embedding from Theorem \ref{tembedding}.
In particular we can then apply the spatial Markov property
of the underlying Poisson process in a smooth and rigorous way, to establish
the Gibbs property of the marginals and hence to add the
arguments missing in \cite{HTHou17}. In fact this is the main
part of the proof.
Even though our coupling does not require the local stability assumption (Dom2),
the latter is crucial for our Theorem \ref{tpoapprox}.


Let $W\in\cX_0$ and $\psi,\psi'\in\bN_{W^c}$.
Recursively we define a sequence of 
mappings $Y_1,Y_2,\ldots$
from $\Omega$ to the Borel subsets of $W$,
along with point processes $\chi_1,\chi_2,\ldots$ and $\chi'_1,\chi'_2,\ldots$.
Given such sequences we define, for $n\in\N$,
$W_n:=Y_1\cup\cdots\cup Y_n$, $\xi_n:=(\chi_n)_{Y_n}$
and $\xi'_n:=(\chi'_n)_{Y_n}$.
The recursion starts with $Y_1:=\{x\in W: x\sim (\psi+\psi')\}$
and
\begin{align*}
(\chi_1,\chi'_1):=(T_{W,\psi}(\Phi_{W\times\R_+}),T_{W,\psi'}(\Phi_{W\times\R_+})).
\end{align*}
The latter definition uses the ordering induced by a Borel isomorphism.
We will modify the original isomorphism by first assuming that
$\varphi(W)\subset (0,1]$ and then setting
\begin{align}\label{varphi1}
\varphi_1:=\frac{\varphi_{Y_1}}{2}+\Big(\frac{1+\varphi}{2}\Big)_{W\setminus Y_1},
\end{align}
where the lower set index denotes restriction.
Then $\varphi_1$ is a Borel isomorphism and each point in $Y_1$ 
is smaller than every point in $W\setminus Y_1$.
This fact will become relevant later on.
For $n\in\N$ we define
\begin{align}\label{rec1}
Y_{n+1}:=\begin{cases}
W\setminus W_n,&\text{if $\chi_n(Y_n)+\chi'_n(Y_n)=0$,}\\
\{x\in W\setminus W_n: x\sim (\chi_n+\chi'_n)_{Y_n}\},
&\text{if $\chi_n(Y_n)+\chi'_n(Y_n)>0$,}
\end{cases}
\end{align}
and
\begin{align}\label{rec2}
(\chi_{n+1},\chi'_{n+1}):=(T_{W\setminus W_n,\psi+\xi_1+\cdots+\xi_n}(\Phi_{(W\setminus W_n)\times\R_+}),
T_{W\setminus W_n,\psi'+\xi'_1+\cdots+\xi'_n}(\Phi_{(W\setminus W_n)\times\R_+})).
\end{align}
Here we use the Borel isomorphism
\begin{align}\label{varphin}
\varphi_{n+1}:=\frac{\varphi_{Y_{n+1}}}{2}+\Big(\frac{1+\varphi}{2}\Big)_{W\setminus W_{n+1}}.
\end{align}
Since $\varphi_{n+1}$ depends on $Y_{n+1}$ it is a random mapping
and in fact a mapping from $\Omega\times (W\setminus W_n)$ to $(0,1]$.
Note that if $\chi_n(Y_n)+\chi'_n(Y_n)=0$ then
$Y_1\cup\cdots\cup Y_{n+1}=W$. In that
case we have $\xi_{n+1}=\xi'_{n+1}$ (see the final part
of the proof of Theorem \ref{tdis}) and $Y_i=\emptyset$ for $i\ge n+2$.

Let us briefly record some measurability properties. 
By the measurability property of $\sim$ it follows that
$Y_1$ is measurable, while the measurability of $\chi_1,\chi'_1$
(as functions on $\Omega$) follows from
the measurability of $T_{W,\psi}$ and $T_{W,\psi'}$.
Therefore $\xi_1,\xi_2$ are measurable as well.
It follows that $Y_2$ is {\em graph-measurable}, that is
$(x,\omega)\mapsto \I\{x\in Y_2(\omega)\}$ is measurable on
$W\times \Omega$. Using the explicit representation
\eqref{eTmeas} for $T_{W\setminus W_1,\kappa_{\psi+\xi_1+\cdots+\xi_n}}$
we see that $\chi_2$ is measurable. Proceeding this way,
it follows inductively that each $Y_n$ (and hence each $W_n$)
is graph-measurable and that $\chi_1,\chi_2,\ldots$, $\chi'_1,\chi'_2,\ldots$,
$\xi_1,\xi_2,\ldots$ and $\xi'_1,\xi'_2,\ldots$ are measurable
and hence point processes.

Define
\begin{align*}
\xi:=\sum^\infty_{n=1}\xi_n,\quad \xi':=\sum^\infty_{n=1}\xi'_n.
\end{align*}
and note that
\begin{align*}
|\xi-\xi'|=\sum^\infty_{n=1}|\xi_n-\xi'_n|,
\end{align*}
where the definition
of the total variation measure $|\mu-\mu'|$ for $\mu,\mu'\in\bN$
has been given at the end of Section \ref{secGibbs}.

\begin{theorem}\label{tdis} Assume that $\kappa$ satisfies (Loc1).
Suppose that $W\in\cX_0$ and 
let $\psi,\psi'\in \bN_{W^c}$ be two regular boundary conditions for $W$.
Construct $\xi$ and $\xi'$ as above.  Then $\xi$ is a Gibbs process on
$W$ with PI $\kappa_{W,\psi}$ while $\xi'$ is a Gibbs process on $W$
with PI $\kappa_{W,\psi'}$.  Every point in $|\xi-\xi'|$ is connected
via $\xi+\xi'$ to $\psi+\psi'$. 
Moreover, if $\kappa$ satisfies (Dom2), then the support of
$\xi+\xi'$ is contained in the support of 
$\int\I\{x\in\cdot\cap W,t\le \alpha(x)\}\,\Phi(\mathrm{d}(x,t))$
which is a Poisson process with intensity measure $\alpha\lambda_W$.
\end{theorem}

\begin{remark} \rm Theorem \ref{tdis} applies for all
boundary conditions to the Strauss process, the continuum random
cluster model and the Widom-Rowlinson model; see Section
\ref{secass}. Moreover, it applies to all boundary conditions
satisfying \eqref{eZS} in a Gibbs process with a pair
potential. For the area interaction process, Theorem \ref{tdis}
applies if $\mathbb{Q}([r_1,\infty))=1$ for some $r_1>0$.
\end{remark} 

\begin{proof}[Proof of Theorem \ref{tdis}] We need some more notation. For each measurable set $B\subset W$
and each $\mu \in\bN_{B^c}$
define $p_{B,\mu}$ by \eqref{ethprob} with $(\BX,\kappa)$
replaced by $(B,\kappa_{B,\mu})$. Define
\begin{align}\label{eyn*}
Y^*_n:=\{(x,t)\in\BX\times\R_+:x\in Y_n,t\le p_n(x)\},
\end{align}
where 
\begin{align*}
p_n(x):=\max\{p_{W\setminus W_{n-1},\psi+\xi_1+\cdots+\xi_{n-1}}(x,\chi_n),
p_{W\setminus W_{n-1},\psi'+\xi'_1+\cdots+\xi'_{n-1}}(x,\chi'_n)\}
\end{align*}
and $W_0:=\emptyset$. Since $Y_n$ is graph-measurable, so is $Y_n^*$.
Set $S_n:=Y^*_1\cup\cdots \cup Y^*_n$.
A crucial tool for our proof is the spatial Markov property
\begin{align}\label{estoppingset}
\BP(\Phi_{(W\times\R_+)\setminus S_n}\in\cdot\mid \Phi_{S_n})
=\Pi_{(\lambda\otimes\lambda_1)_{(W\times\R_+)\setminus S_n}},\quad \BP\text{-a.s.\ on $\{\eta(S_n)<\infty\}$.}
\end{align}
This follows from \cite[Theorem A.3]{LPY2021}, once we will have proved 
that $S_n$ is a {\em stopping set}. The latter means that
$S_n$ is graph-measurable and
\begin{align}\label{estopset}
S_n(\Phi_{S_n}+\mu_{(W\times\R_+)\setminus S_n})=S_n(\Phi),\quad \mu\in \bN^*_{\ell}(W\times\R_+).
\end{align}
Here we (slightly) abuse our notation by interpreting 
$S_n$ as a mapping on $\bN^*_\ell(W\times\R_+)$.

To check \eqref{estopset}, we prove inductively that
$\xi^*_n:=(\xi_1,\xi_1',\ldots,\xi_n,\xi_n')$ and
$Y^*_1,\ldots,Y^*_n$ do not change if the points
in $\Phi_{(W\times\R_+)\setminus S_n}$ are replaced by an arbitrary configuration.
For $n=1$ this follows from \eqref{varphi1} and the definitions of
$T_{W,\psi}(\Phi_{W\times\R_+})$ and $T_{W,\psi'}(\Phi_{W\times\R_+})$. 
Suppose that it is true for some $n\in\N$. In particular,
changing the points of $\Phi$ in $(W\times\R_+)\setminus S_{n+1}$
does not change $\xi^*_n$ and hence also not $Y_{n+1}$.
Hence it follows from \eqref{varphin} that $(\xi_{n+1},\xi_{n+1}')$
does not change either. By definition \eqref{eyn*}
$Y^*_{n+1}$ does not change.

Since $\psi$ and $\psi'$ are regular boundary conditions for $W$
it follows from Theorem \ref{tembedding} that 
$\chi$ (resp.\ $\chi'$) is a Gibbs process with PI $\kappa_{W,\psi}$
(resp.\ $\kappa_{W,\psi'}$). We shall prove by induction that
\begin{align}\label{e5.08}
\chi_{n+1}+\sum^n_{m=1}\xi_m\overset{d}{=}\chi_1,\,
\chi'_{n+1}+\sum^n_{m=1}\xi'_m\overset{d}{=}\chi'_1,
\quad n\in\N_0,
\end{align}
and
\begin{align}\label{e5.18}
\BP(\Phi(S_{n+1})<\infty)=1.
\end{align}

The case $n=0$ of \eqref{e5.08} is trivial.
Since $\psi$ and $\psi'$ are regular boundary conditions
we can recall from the proof of Theorem \ref{tembedding}
that $\int p_W(x,\psi)\,\lambda(\mathrm{d}x)+\int p_W(x,\psi')\,\lambda(\mathrm{d}x)<\infty$.
We assert that the first coordinate of each point of $\Phi_{W\times\R_+}$ in $Y^*_1$ 
is either a point of $\chi_1$ or $\chi_1'$ (or both).
To see this let $(x,t)\in \Phi_{Y_1\times\R_+}$ such that
$p_{W,\psi}(x,\chi_1)\le t$. Suppose that
$\xi_1=\delta_{x_1}+\cdots+\delta_{x_n}$, where $n\in\N_0$
and $x_1<\cdots<x_n$.
By definition \eqref{ex1psi} of the smallest point of $T_{\kappa_W,\psi}$ we must then have that
$\xi_1=(\chi_1)_{Y_1}\ne 0$  and $x\ge x_1$.
From the recursion \eqref{exnpsi} we can also exclude the case $x>x_n$.
So either $n=1$ and $x=x_1$ or 
$n\ge 2$ and there exists $m\in\{1,\ldots,n-1\}$ such that
$x_m<x\le x_{m+1}$. In the latter case we have by definition
of $p_{W,\psi}$ (see \eqref{ethprob}) that
$p_{W,\psi}(x,\chi_1)=p_{W,\psi}(x,\delta_{x_1}+\cdots+\delta_{x_m})$.
Therefore we obtain again from the recursion \eqref{exnpsi}
that $x_{m+1}\le x$ 
and hence $x=x_{m+1}$. This proves our (auxiliary) assertion. Hence we can conclude that
\begin{align}\label{e5.79}
\Phi(S_1)\le \chi_1(W_1)+\chi'_1(W_1),\quad \BP\text{-a.s.} 
\end{align}
which is finite. 

Let $n\in\N$ and assume that \eqref{e5.08} and \eqref{e5.18}
hold for $n-1$. Let $f\colon \bN_W\to\R_+$ be measurable. We have that
\begin{align*}
I:=\BE[f(\xi_1+\cdots+\xi_n+\chi_{n+1})]
=\BE\big[\BE\big[f(\xi_1+\cdots+\xi_n
+T_{W\setminus W_n,\kappa_{\psi+\xi_1+\cdots+\xi_n}}(\Phi_{S^c_n}))\mid \Phi_{S_n}\big]\big],
\end{align*}
where $S_0:=\emptyset$ and $S^c_n:=(W\times\R_+)\setminus S_n$.
We have seen above that 
$$
(\omega,\mu)\mapsto (\xi_1(\omega),\ldots,\xi_n(\omega),
T_{W\setminus W_n(\omega),\kappa_{\psi+\xi_1(\omega)+\cdots+\xi_n(\omega)}}(\mu))
$$
can be written as a function of $(\Phi_{S_n(\omega)}(\omega),\mu)$. As at \eqref{eTmeas}
it can be shown that this function 
is in fact measurable. To deal with $I$ we shall use
\eqref{estoppingset}, Theorem \ref{tembedding} and a standard property of conditional
expectations; see \cite[Theorem 6.4]{Kallenberg}.

In order to apply Theorem \ref{tembedding} we need to check
that $\xi_1+\cdots+\xi_n+\psi$ is almost surely a regular
boundary condition for $W\setminus W_n$.
By assumption \eqref{e62.3} there exists a measurable $H\subset\bN_W$ such that
$\Pi_\lambda(H)=1$ and
\begin{align*}
H\subset\{\mu\in N_W:\text{$Z_B(\mu_{W\setminus B})+\psi)<\infty$ for each measurable $B\subset W$}\}.
\end{align*}
By induction hypothesis $\xi_1+\cdots+\xi_n$ is the restriction
of the Gibbs process $\xi_1+\cdots+\xi_{n-1}+\chi_n$ to $W_n$.
Hence we obtain from Lemma \ref{l5.3} below that
$$
\BP(\xi_1+\cdots+\xi_{n}+\mu_{W\setminus W_n}\in H)=1,\quad \Pi_\lambda\text{-a.e.\ $\mu$}.
$$
And if $\xi_1+\cdots+\xi_{n}\in H$ then
$\xi_1+\cdots+\xi_{n}+\psi$ 
is indeed a regular boundary condition for
$W\setminus W_n$. 

Now we are allowed to apply \eqref{estoppingset} and Theorem \ref{tembedding}
to obtain that 
\begin{align*}
I=\BE\bigg[\int f(\xi_1+\cdots+\xi_n+\mu')\,\Pi_{W\setminus W_n,\psi+\xi_1+\cdots+\xi_n}(\mathrm{d}\mu')\bigg],
\end{align*}
where we recall the definition \eqref{eGibbsmeasure} of a Gibbs measure.
Assume that $n=1$. Since $\chi_1$ is Gibbs we obtain that
\begin{align}\label{e5.44}
I=\iint f(\mu_{Y_1}+\mu')\,\Pi_{W\setminus Y_1,\psi+\mu_{Y_1}}(\mathrm{d}\mu')
\,\Pi_{W,\psi}(\mathrm{d}\mu).
\end{align}
Using \eqref{condDLR} in \eqref{e5.44} with $(B,C)=(W,Y_1)$ we obtain that
$I=\int f(\mu)\,\Pi_{W,\psi}(\mathrm{d}\mu)$, that is, the first part
of \eqref{e5.08} for $n=1$.
Assume now that $n\ge 2$. Define the event
$A_n:=\{(\chi_{n-1}+\chi'_{n-1})(Y_{n-1})=0\}$. 
Then we have that $I=I_1+I_2$, where
\begin{align*}
I_1:=\BE\bigg[\I_{A_n}
\int f(\xi_1+\cdots+\xi_n+\mu')\,\Pi_{W\setminus W_n,\psi+\xi_1+\cdots+\xi_n}(\mathrm{d}\mu')\bigg],
\end{align*}
and $I_2$ is defined in the obvious way. 
On the event $A_n$ we have that $Y_n=W\setminus W_{n-1}$,
$W_n=W$ and $\xi_n=\chi_n$. Therefore,
\begin{align*}
I_1&=\BE [\I_{A_n}f(\xi_1+\cdots+\xi_n)]\\
&=\BE\bigg[\I_{A_n} \int f(\xi_1+\cdots+\xi_{n-1}+\mu')
\,\Pi_{W\setminus W_{n-1},\psi+\xi_1+\cdots+\xi_{n-1}}(\mathrm{d}\mu)\bigg],
\end{align*}
where we have again used \eqref{estoppingset} and Theorem \ref{tembedding}.
On the event $\Omega\setminus A_n$ we have that $\xi_n=(\chi_n)_{Y_n}$.
Therefore we obtain from \eqref{estoppingset} and Theorem \ref{tembedding}
\begin{align}\label{e6.7}\notag
I_2=\BE\bigg[\I_{\Omega\setminus A_n}
\iint & f(\xi_1+\cdots+\xi_{n-1}+\mu_{Y_n}+\mu')\\
&\times \Pi_{W\setminus W_n,\psi+\xi_1+\cdots+\xi_{n-1}+\mu_{Y_n}}(\mathrm{d}\mu')
\,\Pi_{W\setminus W_{n-1},\psi+\xi_1+\cdots+\xi_{n-1}}(\mathrm{d}\mu)\bigg].
\end{align}
Applying \eqref{condDLR} to the right-hand side of \eqref{e6.7}
with $B:=W\setminus W_{n-1}$ and $C:=W\setminus Y_{n}$ yields
\begin{align*}
I_2=\BE\bigg[\I_{\Omega\setminus A_n}\int f(\xi_1+\cdots+\xi_{n-1}+\mu)
\,\Pi_{W\setminus W_{n-1},\psi+\xi_1+\cdots+\xi_{n-1}}(\mathrm{d}\mu)\bigg]
\end{align*}
and hence
\begin{align*}
I&=\BE\bigg[ \int f(\xi_1+\cdots+\xi_{n-1}+\mu)
\,\Pi_{W\setminus W_{n-1},\psi+\xi_1+\cdots+\xi_{n-1}}(\mathrm{d}\mu)\bigg]\\
&=\BE[f(\xi_1+\cdots+\xi_{n-1}+\chi_n],
\end{align*}
where the second equality comes again from Theorem \ref{tembedding}.
This shows the first part of \eqref{e5.08} for $n\ge 2$. Of course
the second part follows in the same way. 

Since $\xi_1+\cdots+\xi_n+\psi$ is almost surely a regular
boundary condition for $W\setminus W_n$ we obtain
exactly as at 
\eqref{e5.79} that
\begin{align}
  \Phi(S_{n+1})\le \chi_{n+1}(W_{n+1})+\chi'_{n+1}(W_{n+1}),\quad \BP\text{-a.s.} 
\end{align}
This shows \eqref{e5.18} and finishes the inductive proof of 
\eqref{e5.08} and \eqref{e5.18}. We shall only use \eqref{e5.08}.


For each $m\in\N$ we have that
\begin{align*}
\BP(\xi(W)>k)=\lim_{n\to\infty}\BP((\xi_1+\cdots+\xi_n)(W)>k)\le \BP(\chi_1>k).
\end{align*}
Therefore $\BP(\xi(W)<\infty)=1$,
so that almost surely $\xi_n(W)=0$ for all sufficiently
large $n$. Since the same holds for the point processes $\xi'_n$ we have by definition of
the recursion that $\chi_n=0$ for large enough $n$.
Therefore we obtain for each bounded measurable $f\colon \bN_W\to\R$
by bounded convergence
\begin{align}
  \BE[f(\xi)]=\lim_{n\to\infty}\BE\bigg[f\bigg(\sum^n_{m=1}\xi_m\bigg)\bigg]
=\lim_{n\to\infty}\BE\bigg[f\bigg(\chi_{n+1}+\sum^n_{m=1}\xi_m\bigg)\bigg],
\end{align}
so that $\xi$ is Gibbs with PI $\kappa_\psi$.

Finally we let $n\in\N$ be the smallest integer
such that $\chi_n(Y_n)+\chi'_n(Y_n)=0$. Then $Y_{n+1}=W\setminus W_n$
and (Loc1) implies that
\begin{align*}
\kappa_{{W\setminus W_n},\psi+\xi_1+\cdots+\xi_n}
=\kappa_{{W\setminus W_n},\psi'+\xi'_1+\cdots+\xi'_n}.
\end{align*}
Hence $\xi_{n+1}=\xi'_{n+1}$ and $\xi_i=\xi'_i=0$ for $i\ge n+2$.
Therefore each point from $|\xi-\xi'|$ must lie in $W_n$.  By
definition all those points are connected via $\xi_{W_n}+\xi'_{W_n}$
to $\psi+\psi'$. The final assertion follows from
the definition \eqref{rec2} and Lemma \ref{l4.1}. 
\end{proof}

\bigskip
The following lemma has been used in the preceding proof.

\begin{lemma}\label{l5.3} Let $W\in\cX_0$ and suppose that 
$\xi$ is a Gibbs process on $W$. Let $S$ be a graph-measurable
mapping from $\Omega$ into $\cX\cap W$.
Then 
$\int\BP(\xi_S+\mu_{W\setminus S}\in\cdot)\,\Pi_\lambda(\mathrm{d}\mu)
\ll \Pi_{\lambda_W}$.
\end{lemma}
\begin{proof} Let $f\colon\bN\to\R_+$ be measurable such that
$\int f(\mu)\,\Pi_{\lambda_W}(\mathrm{d}\mu)=0$.
Since a Poisson process is completely independent 
and $\lambda(S)\le \lambda(W)<\infty$ we have
\begin{align*}
I&:=\BE\bigg[\int f(\xi_S+\mu_{W\setminus S})\,\Pi_\lambda(\mathrm{d}\mu)\bigg]\\
&=\BE\bigg[e^{\lambda(S)}
\int \I\{\mu(S)=0\}f(\xi_S+\mu_{W\setminus S})\,\Pi_\lambda(\mathrm{d}\mu)\bigg]
&\le e^{\lambda(W)}\BE\bigg[
\int f(\xi_S+\mu_{W})\,\Pi_\lambda(\mathrm{d}\mu)\bigg].
\end{align*}
By \cite[Theorem 1.1]{HoSoo13}
(applying to general Borel spaces),
$\BP(\xi_S\in\cdot)\ll \BP(\xi\in\cdot)$. Hence
we also have $\BP(\xi_S\in\cdot)\ll \Pi_{\lambda_W}$
and we let $g\colon \bN\to\R_+$ denote the corresponding density.
Then
\begin{align*}
I\le e^{\lambda(W)}\iint g(\psi_W)f(\psi_W+\mu_{W})\,\Pi_\lambda(\mathrm{d}\mu)\,\Pi_\lambda(\mathrm{d}\psi).
\end{align*}
Noting that
\begin{align*}
\iint f(\psi_W+\mu_{W})\,\Pi_\lambda(\mathrm{d}\mu)\,\Pi_\lambda(\mathrm{d}\psi)
=\int f(\mu_W)\,\Pi_{2\lambda}(\mathrm{d}\mu)
\end{align*}
and that $\Pi_{2\lambda_W}$ and $\Pi_{\lambda_W}$ are equivalent, 
we obtain that $I=0$. This concludes the proof.
\end{proof}

\begin{remark}\label{rmeasurability}\rm Consider the assumptions
of Theorem \ref{tdis}. We assert that the mapping
$(\omega,\psi,\psi')\mapsto (\xi(\omega),\xi'(\omega))$ is measurable.
Since measurability issues can be a little tricky
at times, we give here an explicit argument.
Let $\Psi$ be a point process on $\BX$ and $Z$ a
graph-measurable mapping from $\Omega$ into $\cX_0$.
Let $B\in\cX$. By \eqref{eTmeas} we have that
\begin{align*}
&T_{Z,\Psi}(\Phi_{Z\times\R_+})(B)=
\sum^\infty_{n=1}\int\I\{t_1\le p(x_1,\Psi),\ldots,t_n\le p(x_n,\Psi+\delta_{x_1}+\cdots+\delta_{x_{n-1}})\}\\
&\qquad\times\I\{x_1<\cdots<x_n\}\I\{x_1,\ldots,x_n\in Z\}\I\{(x_1,\ldots,x_{n-1})\in C_{n-1}\}\\
&\qquad\times\I\{\Phi(B_{n}(x_1,\ldots,x_n)\cap (Z\times\R_+))=0\}
\sum^n_{i=1} \I\{x_i\in B\}\,\Phi^{(n)}(\mathrm{d}(x_1,t_1,\ldots,x_n,t_n)).
\end{align*}
Here the order $<$ is allowed to depend measurably on $\omega\in\Omega$.
Writing $R_\omega$ for this order, the measurability  means that
$(\omega,x,y)\mapsto \I\{x R_\omega y\}$ is measurable.
It follows that $T_{Z,\Psi}(\Phi_{Z\times\R_+})(B)$ is a random variable which in turn
implies that $T_{Z,\Psi}(\Phi_{Z\times\R_+})$ is a point process.
Using this fact together with the recursive construction of $(\xi,\xi')$
shows that $(\xi,\xi')$ is indeed jointly measurable in $\omega\in\Omega$ and
the boundary conditions $\psi$ and $\psi'$.
\end{remark}

\section{Bounds for empty space probabilities}\label{sempty}

In this section we apply the previous results to obtain upper bounds
for empty space probabilities of a Gibbs process. 
These probabilities are important chracteristics of
a point process. In fact, they determine the distribution of a simple
point process (see e.g.\ \cite{LastPenrose17}) and have
many applications, for instance in stochastic geometry (see e.g.\ \cite{CSKM13})
or in quantifying the clustering of a point
process (see \cite{BY14}). 
We shall need an upper bound in the proof of Theorem \ref{hardtpo}.

A naive approach to bound the empty space
probability for a given boundary condition from above would be to
bound the partition function from below and to exploit the relation
\eqref{eempty}. This is often not very promising since the partition
function is hardly accessible. Our approach is different. We start
with the embedding representation of a (finite) Gibbs process from
Section 5 that involves the function $p$ from \eqref{ethprob}. This
enables us to use bounds for fractions of partition functions (which
are often considerably easier to find than bounds for the partition
function itself) to obtain bounds for empty space probabilities.

Let $(\BX,\cX)$ be a Borel space equipped with a $\sigma$-finite and
diffuse measure $\lambda$.
We consider a Gibbs process $\xi$ on 
$\BX$ with PI $\kappa$ satisfying the cocycle condition \eqref{ecocycle}
and assume that (Dom2) and (Loc1) hold.

Let $\eta$ be a Poisson process
on $\BX$ with intensity measure $\alpha\lambda$
and let $B\in\cX_0$.
By \cite[Theorem 1.1]{GeorKun97} (and \cite[Lemma 2.5]{BHLV20})
the process $\xi_B$,
conditioned on $\xi_{B^c}$, is stochastically dominated by
$\eta$. In particular
\begin{align}\label{hardlowbou}
\BP(\xi(B)=0 \mid \xi_{B^c})\ge e^{-(\alpha\lambda)(B)},\quad \BP\text{-a.s.} 
\end{align}
The next result provides an inequality in the converse direction.
For $B\in\cX_0$ we set 
\begin{align*}
B_\xi:=\{x\in B: x\not\sim \xi_{B^c}\}=B\setminus \bigcup_{y\in\xi_{B^c}} N_y.
\end{align*}

\begin{theorem}\label{tempty1} Suppose that $\xi$ is a Gibbs process whose
PI $\kappa$ satisfies (Dom2) and (Loc1).  
Let $B\in\cX_0$. Then   
\begin{align}\label{harduppbou}
  \BP(\xi(B)=0 \mid \xi_{B^c})
  \le \exp\bigg[-\int \I\{x\in B_\xi\}e^{-(\alpha\lambda)(N_x\cap B)}\kappa(x,0)\,\lambda(\mathrm{d}x)\bigg],
\quad \BP\text{-a.s.}
\end{align}
\end{theorem}
{\em Proof:} We have mentioned at the end of Section \ref{secGibbs}
that the conditional distribution $\BP(\xi_B\in\cdot\mid \xi_{B^c})$
is that of a Gibbs process with PI $\kappa_{B,\xi_{B^c}}$.
By assumption (Dom2) and Remark \ref{rbounded}
we can apply Theorem \ref{tembedding} (with $\kappa=\kappa_{B,\xi_{B^c}}$
and $\BX=B$) to construct this Gibbs distribution.
By definition of the embedding operator and the properties
of the underlying Poisson process we have that
\begin{align}\label{eemptyspace}
\BP(\xi(B)=0 \mid \xi_{B^c})
=\exp\bigg[-\int \I\{x\in B\}p_{B,\xi_{B^c}}(x,0)\,\lambda(\mathrm{d}x)\bigg],\quad \BP\text{-a.s.}
\end{align}
Recall from \eqref{ethprob} that
\begin{align*}
p_{B,\xi_{B^c}}(x,0)=\kappa(x,\xi_{B^c}) 
\frac{Z_{B_{x}}(\xi_{B^c}+\delta_{x})}{Z_{B_x}(\xi_{B^c})},
\quad x\in B,
\end{align*}
where $B_x=B\cap(x,\infty)$ and the intervals are defined with respect to the order
(induced by the Borel isomorphism) on $\BX$.
Assume that $x\in B_\xi$.
If  $\mu\in\bN_B$ satisfies $\mu(N_x)=0$, then assumption (Loc1)
shows that
\begin{align}\label{ekillx}
H(\mu,\xi_{B^c}+\delta_{x})=H(\mu,\xi_{B^c}).
\end{align}
Therefore
\begin{align*}
Z_{B_x}(\xi_{B^c}+\delta_{x})
&\ge \int \I\{\mu(N_x)=0\}e^{-H(\mu,\xi_{B^c})}\,\Pi_{\lambda_{B_x}}(\mathrm{d}\mu),
\end{align*}
Since a Poisson process is completely independent
we obtain that
\begin{align*}
Z_{B_x}(\xi_{B^c}+\delta_{x})&\ge 
e^{-\lambda(N_x\cap B_x)} \int e^{-H(\mu_{N_x^c},\xi_{B^c})}\,\Pi_{\lambda_{B_x}}(\mathrm{d}\mu)\\
&= e^{-\lambda(N_x\cap B_x)} Z_{B_x\setminus N_x}(\xi_{B^c}).
\end{align*} 
Using the definition of the Hamiltonian we further obtain that
\begin{align*}
Z_{B_x}(\xi_{B^c})
&=\int e^{-H(\mu_{N_x^c},\xi_{B^c})} e^{-H(\mu_{N_x},\xi_{B^c}+\mu_{N_x^c})} 
\,\Pi_{\lambda_{B_x}}(\mathrm{d}\mu)\\
&\le \int e^{-H(\mu_{N_x^c},\xi_{B^c})} \prod_{y\in\mu_{N_x}}\alpha(y) 
\,\Pi_{\lambda_{B_x}}(\mathrm{d}\mu),
\end{align*}
where we have used our assumption (Dom2).
By the independence properties of a Poisson process this implies
\begin{align*}
Z_{B_x}(\xi_{B^c})
&\le \exp\bigg[\int \I\{y\in N_x\cap B_x\} (\alpha(y)-1)\,\lambda(\mathrm{d}y)\bigg]
Z_{B_x\setminus N_x}(\xi_{B^c}).
\end{align*}
Therefore we obtain for $x\in B_\xi$ 
that 
\begin{align*}
p_{B,\xi_{B^c}}(x,0)
&\ge \exp\bigg[-\lambda(N_x\cap B_x)+\int \I\{y\in N_x\cap B_x\} (1-\alpha(y))\,\lambda(\mathrm{d}y)\bigg]
\kappa(x,0)\\
&\ge e^{-(\alpha\lambda)(N_x\cap B)}\kappa(x,0).
\end{align*}
Inserting this into  \eqref{eemptyspace} gives the result. \qed

\bigskip

\begin{remark}\rm If $N_x=\{x\}$ for each $x\in\BX$, then
(Loc1) means that $\xi$ is a Poisson process with
intensity measure $\kappa(x,0)\lambda(\mathrm{d}x)$. 
Then \eqref{harduppbou} is an identity. In a sense
this boundary case is obtained in the limit as $N_x\downarrow\{x\}$
for each $x\in\BX$.
\end{remark}

\begin{remark}\rm It is interesting to note
that the preceding theorem holds for any reference measure $\lambda$.
In particular they do not require 
the clusters $C(x,\eta)$ (or $C(x,\xi)$) to be finite.
This is in contrast to the assumptions of Lemma 3.3
in the seminal paper \cite{SchrYuk13}, which applies only
to a restricted range of parameters.
\end{remark}

In the following example we apply Theorem \ref{tempty1} to general Gibbs
particle processes with deterministically bounded grains.

\begin{example}\label{expartic} \rm 
  Let $\mathbb{X}:= \mathcal{C}^{d}$ denote the space of compact and
  nonempty subsets (particles) of $\R^d$. We equip $\mathcal{C}^{d}$
  with the Hausdorff metric and the associated Borel $\sigma$-field
  $\mathcal{B}(\mathcal{C}^d)$. For $K \in \mathcal{C}^d$ let $z(K)$
  denote the center of the circumscribed sphere of $K$. Define
\begin{align*}
\lambda(\cdot):=\iint \I\{K+x \in \cdot\}\mathbb{Q}(\mathrm{d}K) \lambda_d(\mathrm{d}x),
\end{align*}
where $\lambda_d$ denotes Lebesgue measure on $\R^d$ and $\mathbb{Q}$
is a probability measure on $\mathcal{C}^{d}$ satisfying
\begin{align*}
\mathbb{Q}(\{K \in \mathcal{C}^d:C \subset B(o,R)\})=1
\end{align*}
for some fixed $R>0$, where $o$ denotes the origin in $\mathbb{R}^d$. Assume that (Dom2) holds.
Assume that (Loc1) holds
with respect to the relation $\sim$ defined by 
$K \sim L$ if $K \cap L \neq \emptyset$.
Finally we assume the translation invariance
$\kappa(K,\mu)=\kappa(K-x,\theta_x\mu)$ for $K \in \mathcal{C}^d$,
$x \in \R^d$ and $\mu \in \mathbf{N}$, where $\theta_x\mu$ is defined by
$\theta_x\mu(B):=\mu(B+x)$ for measurable
sets $B\subset\mathbb{X}$.
Let $\mathcal{C}^d_t:=\{K \in \mathcal{C}^d:\,z(K) \in B(o,t) \}$ for
$t\ge 2R$. Then Theorem \ref{tempty1} implies that
\begin{align}\label{harduppbou3}
\BP(\xi(\mathcal{C}^d_t)=0 \mid \xi_{B^c})\le e^{-c_0(t-2R)^d},\quad \BP\text{-a.s.},\,t\ge 2R,
\end{align}
where 
\begin{align} \label{exPartc0}
 c_0:=e^{-\alpha\kappa_d 2^d R^d } \int \kappa(K,0)\,\mathbb{Q}(\mathrm{d}K).
\end{align}
If $c_0>0$ (that is if $\int \kappa(K,0)\,\mathbb{Q}(\mathrm{d}K)>0$) then we have 
\begin{align}\label{exPoilik}
 \limsup_{t \to \infty}  t^{-d} \log \BP(\xi(\mathcal{C}^d_t)=0 \mid \xi_{B^c})\le -c_0\quad \BP\text{-a.s.},
\end{align}
\end{example}

\begin{remark} \rm
The authors of \cite{SchrYuk13}  called the property \eqref{exPoilik} {\em Poisson-like}
and established it under assumptions similar to those in
Corollary \ref{cpoapprox}. Our method does not require the absence of percolation
or a related subcritcality property.
\end{remark}

\bigskip

In the following result we replace the assumption (Loc1)
by the stronger assumption (Loc2).

\begin{theorem}\label{tempty2} Suppose that $\xi$ is a Gibbs process whose
PI $\kappa$ satisfies (Dom2) and (Loc2). 
Let $B\in\cX_0$. Then 
\begin{align}\label{harduppbou2}
\BP(\xi(B)=0 \mid \xi_{B^c})
\le \exp\bigg[-\int \I\{x\in B\}e^{-(\alpha\lambda)(N_x\cap B)}\kappa(x,\xi_{B^c})\,\lambda(\mathrm{d}x)\bigg],
\quad \BP\text{-a.s.}
\end{align}
\end{theorem}
{\em Proof:} Let $x\in B$ and $\mu\in N_B$ such that
$\mu(N_x)=0$. By assumption (Loc2) the identity
\eqref{ekillx} remains true,
so that the proof of Theorem \ref{tempty1}  applies. \qed

\begin{corollary}\label{cempty2} Let the assumptions of Theorem \ref{tempty2}
  be satisfied. Assume moreover that there exists a measurable
$\kappa'\colon \BX\times\bN\to\R_+$ such that $\kappa\ge \kappa'$ and
$\kappa'(x,\cdot)$ is decreasing for each $x\in\BX$. Then
\begin{align}\label{e987}
\BP(\xi(B)=0)
\le \BE \exp\bigg[-\int \I\{x\in B\}e^{-(\alpha\lambda)(N_x\cap B)}\kappa'(x,\eta_{B^c})\,\lambda(\mathrm{d}x)\bigg].
\end{align}
\end{corollary}
{\em Proof:} 
The point process $\xi_{B^c}$ is stochastically dominated by $\eta_{B^c}$. Hence the
result follows upon replacing in \eqref{harduppbou2} $\kappa$ by $\kappa'$
and then taking expectations.
\qed

\begin{example}\label{exdecaystrauss} \rm 
Assume that $\BX:=\R^d\times\R_+$
equipped with the product measure $\lambda:=\lambda_d\otimes\BQ$,
where $\BQ$ is a probability measure on $\R_+$ satisfying
\begin{align}\label{eintQ}
\int s^d\,\BQ(\mathrm{d}s)<\infty.
\end{align}
Assume that $\xi$ is a Gibbs process on $\BX$
with PI $\kappa$ satisfying (Dom2) for some constant $\alpha>0$.
We define a relation $\sim$ on $\BX$ by 
\begin{align}\label{esim}
(x,r)\sim (y,s) \Longleftrightarrow \|x-y\|\le r+s
\end{align}
and assume that (Loc2) holds. 
Assume  that there exist $\beta'>0$ and $\beta\in[0,1)$ such that
\begin{align}\label{elowstrauss}
  \kappa(x,r,\mu)\ge \beta'\beta^{\mu(N_{(x,r)})},\quad (x,r,\mu)\in\R^d\times\R_+\times\bN.
\end{align} 
These assumptions are satisfied by the Strauss process discussed
in Example \ref{exstrauss}. 
We apply Corollary \ref{cempty2} and
take $B:=B_t\times\R_+$, where
$B_t:=B(o,t)$ is the ball with radius $t>0$, centred at the
origin $o$. Then
\begin{align}\label{harduppbou4}
  \BP(\xi(B_t\times\R_+)=0)
  \le \BE \exp\bigg[-\iint \I\{x\in B_t\}
  e^{-(\alpha\lambda)(N_{(x,u)}}\beta'\beta^{\eta(N_{(x,r)}\cap (B_t^c\times\R_+))}\,\mathrm{d}x\,\BQ(\mathrm{d}r)\bigg].
\end{align}
We have for all $(x,r)\in\R^d\times\R_+$ that
\begin{align*}
\alpha\lambda(N_{(x,r)})=\alpha\iint\I\{\|y-x\|<r+s\}\,\mathrm{d}y\,\BQ(\mathrm{d}s)
=\alpha \int \kappa_d(r+s)^d\,\BQ(\mathrm{d}s)\le cr^d
\end{align*}
for some $c>0$.   Applying Jensen's inequality to the probability measure
$\frac 1{\kappa_d t^d} \lambda_{B_t}\otimes\BQ$, we find that
\begin{align}\label{ebound7}
\BP(\xi(B_t\times\R_+)=0)
\le \frac{1}{\kappa_d t^d} \BE  \iint \I\{x\in B_t\} \exp\big[-\beta'\kappa_d t^d e^{-cr^d}
\beta^{Z_{t,x,r}}\big]\,\mathrm{d}x\,\BQ(\mathrm{d}r),
\end{align}
where $Z_{t,x,r}$ has a Poisson distribution with parameter
$\alpha\lambda(N_{(x,r)}\cap (B_t^c\times\R_+))$.

We treat \eqref{ebound7} distinguishing by the value of $\beta^{Z_{t,x,r}}$. Let $a\in(0,d)$. Then
\begin{align*}
\frac{1}{\kappa_dt^d}\BE&\iint\I\{t^d\beta^{Z_{t,x,r}}\ge t^a\} \I\{x\in B_t\} \exp\big[-\beta'\kappa_dt^de^{-cr^d}
\beta^{Z_{t,x,r}}\big]\,\mathrm{d}x\,\BQ(\mathrm{d}r)\\
&\le\frac{1}{\kappa_dt^d}\iint\I\{t^d\beta^{Z_{t,x,r}}\ge t^a\} \I\{x\in B_t\} 
\exp\big[-\beta'e^{-cr^d}t^{a}\big]\,\mathrm{d}x\,\BQ(\mathrm{d}r)\\
&\le \int \exp\big[-\beta'e^{-cr^d}t^{a}\big]\,\BQ(\mathrm{d}r).
\end{align*}
By Jensen's inequality this can be bounded by
\begin{align*}
\exp\Big[-\beta't^{a} \int e^{-cr^d}\,\BQ(\mathrm{d}r)\Big].
\end{align*}
We further have that
\begin{align*}
\frac{1}{\kappa_dt^d}\BE&\iint\I\{t^d\beta^{Z_{t,x,r}}\le t^a\} \I\{x\in B_t\} \exp\big[-\beta'\kappa_dt^de^{-cr^d}
\beta^{Z_{t,x,r}}\big]\,\mathrm{d}x\,\BQ(\mathrm{d}r)\\
&\le\frac{1}{\kappa_dt^d}\iint  
\I\{x\in B_t\}\BP(\beta^{-Z_{t,x,r}}> t^{d-a})\mathrm{d}x\,\BQ(\mathrm{d}r).
\end{align*}
By Markov's inequality
\begin{align*}
  \BP(\beta^{-Z_{t,x,r}}> t^{d-a})\le t^{a-d}\BE Z_{t,x,r}=t^{a-d}e^{\BE Z_{t,x,r}(\beta^{-1}-1)}.
\end{align*}
Hence the above integral is bounded by
\begin{align*}
\frac{1}{\kappa_dt^d}\iint  
\I\{x\in B_t\}e^{-c r^d}t^{a-d}\mathrm{d}x\,\BQ(\mathrm{d}r)=t^{a-d} \int e^{- c r^d} \,\BQ(\mathrm{d}r).
\end{align*}
Altogether we obtain polynomial decay of $\BP(\xi(B_t\times\R_+)=0)$.
\end{example}

\begin{remark}\rm The polynomial decay in Example \ref{exdecaystrauss} might be suboptimal.
But is has been derived under the minimal integrability assumption 
\eqref{eintQ}. This in contrast to Example \ref{expartic}, where the
the range of interaction between neighbors has been assumed to be 
deterministically bounded.
\end{remark}

\begin{remark} \rm Since (Dom2) and (Loc2) hold for the Strauss
process, the continuum random cluster model and the
Widom-Rowlinson model, Theorem \ref{tempty1} and Theorem
\ref{tempty2} apply to these models. Moreover, they apply to Gibbs
processes with a pair potential if \eqref{eZS} holds and to the
area interaction process if $\mathbb{Q}([r_1,\infty))=1$ for some  $r_1>0$.
\end{remark} 

\section{Palm measures and thinnings}\label{secPalm}

To prepare our results on Poisson approximation
in the next section we define and briefly discuss
Palm measures and dependent thinnings of a Gibbs process.

Again we work in the general setting of Section \ref{secGibbs}
and let $(\BX,\cX)$ be a Borel space equipped
with a $\sigma$-finite measure $\lambda$. 
Let $\xi$ and $\chi$ be two point processes on $\BX$
and assume that $\chi$ has a $\sigma$-finite intensity
measure $\BE[\chi]$. The {\em Palm distributions} 
$\BP^{\xi|\chi}_x$ (of $\xi$ w.r.t.\ $\chi$), $x\in\BX$, are a family 
of probability measures on $\bN(\BX)$
such that $(x,A)\mapsto \BP^{\xi|\chi}_x(A)$ is a probability
kernel and
\begin{align}\label{ePalm}
\BE\int h(x,\xi)\,\chi(\mathrm{d}x)
=\int h(x,\mu)\,\BP^{\xi|\chi}_x(\mathrm{d}\mu)\,\BE[\chi](\mathrm{d}x)
\end{align}
for all measurable $h\colon\BX\times\bN(\BX)\to [0,\infty)$.
One can interpret $\BP^{\xi|\chi}_x(\cdot)$ as conditional
distribution given that $\chi$ has a point at $x$.
We refer to \cite{Kallenberg17} for the existence
and an in depth discussion.

Let $\xi$ be a Gibbs process on $\BX$ with
PI $\kappa$ satisfying the cocycle condition \eqref{ecocycle}.
Assume given a measurable function $g\colon\BX\times\bN\to\{0,1\}$
satisfying the hereditary property
\begin{align}\label{eher}
\{(y,\mu)\in\BX\times\bN:g(x,\mu+\delta_y)=0\}\subset
\{(y,\mu)\in\BX\times\bN:g(x,\mu)=0\},\quad x\in\BX.
\end{align}
Using $g$ we define a (measurable) thinning operator 
$\Gamma\colon\bN\to\bN$ by
\begin{align}\label{eGammaop}
\Gamma(\mu)(B):=\int\I\{x \in B\}g(x,\mu-\delta_x)\,\mu(\mathrm{d}x),\quad B\in\cX,\,\mu\in\bN.
\end{align}
Since $\Gamma(\xi)\le \xi$ we obtain from definition 
\eqref{ePalm} that
$\BP^{\xi|\Gamma(\xi)}_x(\{\mu\in\bN(\BX):\mu(\{x\}\ge 1\})=1$
for $\BE[\Gamma(\xi)]$-a.e.\ $x\in\BX$.
Later we shall need the following Gibbs property of $\BP^{\xi|\Gamma(\xi)}_x$.

\begin{lemma}\label{lpalm} The probability measure
$\BP^{\xi|\Gamma(\xi)}_x(\{\mu\in\bN:\mu-\delta_x\in\cdot\})$
is for $\BE[\Gamma(\xi)]$-a.e.\ $x\in\BX$
the distribution of a Gibbs process with PI $\kappa^x$ given by
\begin{align*}
\kappa^x(y,\mu):=\kappa(y,\mu+\delta_x)\frac{g(x,\mu+\delta_y)}{g(x,\mu)},
\quad (y,\mu)\in\BX\times\bN,
\end{align*}
where $0/0:=0$.
\end{lemma}
{\em Proof:} By the GNZ equation \eqref{eGNZ}, the intensity
measure of $\Gamma(\xi)$ is given by
\begin{align*}
\BE[\Gamma(\xi)](B)
&=\BE \int\I\{x\in B\}g(x,\xi-\delta_x)\,\xi(\mathrm{d}x)\\
&=\BE \int\I\{x\in B\}g(x,\xi)\kappa(x,\xi)\,\lambda(\mathrm{d}x).
\end{align*}
This measure is $\sigma$-finite. The remainder of the proof
is quite standard; see e.g.\ \cite{BHLV20} for the case  $g\equiv 1$.
\qed

\section{Poisson approximation of Gibbsian functionals}\label{secpoissonapp}

In this section we  let $(\BY,\cY)$ be
a complete separable metric space
equipped with the Borel $\sigma$-field $\cY$ and some
probability measure $\BQ$. We consider the product space $\BX:=\R^d\times\BY$
equipped with the product measure $\lambda:=\lambda_d\otimes\BQ$,
where $\lambda_d$ denotes Lebesgue measure on $\R^d$.
For a (signed) measure $\rho$ on $\BX$ and a Borel set $B\subset\R^d$
we abbreviate $\rho_B:=\rho_{B\times \BY}$ and $\rho(B):=\rho(B\times\BY)$.
We consider a Gibbs process $\xi$ on $\BX$
with PI $\kappa$ satisfying the cocycle condition \eqref{ecocycle}
and (Dom2) for some $\alpha>0$.
Let $\sim$ be a measurable symmetric relation on $\BX$
and assume that (Loc1) holds.

We consider a measurable function $g\colon\BX\times\bN\to\{0,1\}$
satisfying \eqref{eher}. Let $R \subset \mathbb{R}^d$ be
a compact set and let $R(x):=R+x$ for all $x \in \mathbb{R}^d$. We assume that for
all $(x,r,\mu)\in\R^d\times\BY\times\bN(\BX)$,
\begin{align}
	g(x,r,\mu)&=g(x,r,\mu_{R(x)})\label{eygt1}.
\end{align}
Let $W\subset \mathbb{R}^d$ be a compact set. 
We wish to approximate the restriction of the point process
$\Gamma(\xi)$ to $W\times\BY$ by a Poisson process $\nu$
on  $W\times\BY$.

To the best of our knowledge, Poisson approximation of (derived) Gibbs
processes is so far only discussed in very few articles in the
literature. In \cite[Theorem 3.A]{Schuh09} Stein's method (through
\cite[Theorem 2.4]{BB92}) is used to obtain bounds on the total
variation distance between a finite thinned Gibbs process and a
Poisson process. 
However, as \cite[Theorem 4.I]{Schuh09} and the
examples given thereafter show, it is not easy to exploit these
bounds. (One needs to bound distances between densities
and empty space probabilities.)
In Theorem \ref{tpoapprox} and Corollary \ref{cpoapprox} below, we treat a
wider class of (scaled) thinned Gibbs processes and give very explicit
bounds on their total variation distance to a Poisson process. Our
proof exploits \cite[Theorem 3.1]{BSY21} and is based on a coupling of
the thinned Gibbs process and its Palm version. This technique is
applied in \cite{BSY21} and \cite{Otto20} to Poisson approximation of thinned Poisson
processes.

Let $\mu$ be a simple counting measure on $\BX$. We can extend $\bar\mu:=\mu(\cdot\times \BY)$ to a
graph $G(\mu)$ as follows. Let $x,y\in\bar\mu$ be distinct and let
$r,s\in\BY$ such that $(x,r),(y,s)\in\mu$. We draw an edge between
$x$ and $y$ if $(x,r)\sim (y,s)$. For Borel sets $A,B \subset \mathbb{R}^d$ we write $A \xleftrightarrow{\mu}B$ if
there exist $\bar\mu$-points $x\in A$ and $y\in B$ which
belong to the same component of $G(\mu)$.

Recall that the {\em Kantorovich-Rubinstein (KR) distance} between (the distributions) of two
finite point processes $\xi'$ and $\xi''$ on $\BX$ is defined by
\begin{align*}
	\kr(\xi',\xi''):=\sup_{h \in \text{Lip}(\mathbb{X})} |\BE h(\xi')-\BE h(\xi'')|,
\end{align*}
where $\text{Lip}(\mathbb{X})$ is the class of all measurable
1-Lipschitz functions with respect to the total variation metric
$\tv$   
between two (finite) measures on $\mathbb{X}$; see \eqref{e2.141}. Note
that convergence in the KR distance implies convergence in
distribution and that the KR distance dominates the 
total  variation distance
$\tv(\BP(\xi'\in\cdot),\BP(\xi''\in\cdot))$.

In the next theorem we will, in addition to the set $R$ (introduced at
\eqref{eygt1}), consider another compact set $S \subset \mathbb{R}^d$ 
that contains the origin $o \in \mathbb{R}^d$. We define $S(x):=S+x$,
$x \in \mathbb{R}^d$ 
and let  $W+S:=\{x+y:\,x \in W,\,y \in S\}$ be the (compact) Minkowski sum of $W$ and $S$. 
Note that $W \subset W+S$ since $o \in S$.

\begin{theorem}\label{tpoapprox} Let $\xi$ be a Gibbs process on $\R^d\times\BY$
  with a PI $\kappa$ satisfying (Loc1) and (Dom2) for some
  $\alpha>0$. Let $R \subset \mathbb{R}^d$, $S \subset \mathbb{R}^d$ with $o \in S$ 
and $W \subset \mathbb{R}^d$ be compact sets such that
  $R \subset S$. Define $\Gamma$ by
  \eqref{eGammaop}, where $g$ is assumed to satisfy \eqref{eher} and \eqref{eygt1}. 
Let
  $\nu$ be a Poisson process with finite intensity measure
  $\mathbb{E}[\nu]$. Then
\begin{align}\label{epoa}
\kr(\Gamma(\xi)_W,\nu)
\le 
\|\BE[\Gamma(\xi)_W]-\BE[\nu]\|
+ T_1+T_2+T_3,
\end{align}
where
\begin{align*}
	T_1:= &2 \iint \I\{x,y\in W\} \BE [g(x,r,\xi) \kappa(x,r,\xi)] \,\BE [g(y,s,\xi) \kappa(y,s,\xi)]\\
	&\quad \quad \times \I \{S(x) \cap S(y) \neq \emptyset\}\,\BQ^2(\mathrm{d}(r,s))\,\mathrm{d}(x,y),\\
	T_2:= &2 \iint \I\{x,y\in W\} \BE [g(x,r,\xi+\delta_{(y,s)}) \kappa(x,r,\xi+\delta_{(y,s)})g(y,s,\xi+\delta_{(x,r)}) \kappa(y,s,\xi)]\\
	&\quad \quad \times \I \{S(x) \cap S(y) \neq \emptyset\} \,\BQ^2(\mathrm{d}(r,s))\,\mathrm{d}(x,y),\\
	T_3:=&2\alpha \iint \I\{x,y\in W\}\I \{S(x) \cap S(y) = \emptyset\}
	\BP(R(y)\overset{\eta}{\longleftrightarrow}(W+S)^c \cup R(x))\\
	&\quad \quad \times \BE [g(x,r,\xi) \kappa(x,r,\xi)]\, \mathbb{Q}(\mathrm{d}r)\,	\mathrm{d}(x,y),
\end{align*}
where $\eta$ is a Poisson process on $\mathbb{R}^d \times \mathbb{Y}$ with intensity
measure $\alpha \lambda_d\otimes\BQ$.
\end{theorem}
\noindent

\begin{remark} \rm One should compare Theorem \ref{tpoapprox} with
\cite[Theorem 4.1]{BSY21} that considers Poisson process
approximation of functionals of a Poisson process. The terms $T_1$
and $T_2$ on the right-hand side of \eqref{epoa} are analogous to
the terms $E_2$ and $E_3$ in \cite[Theorem 4.1]{BSY21}. 
The term  $T_3$, which reflects the 
interactions of the Gibbs process
$\xi$, does not appear in \cite[Theorem 4.1]{BSY21}. This is due to
the independence properties of the Poisson process.
\end{remark}

\begin{remark}\label{remf} \rm Using the contractivity properties of the 
KR-distance, Theorem \ref{tpoapprox} can be generalized as follows.
Let $\BX'$ be another complete separable metric space
and let $f\colon\mathbb{X} \to\mathbb{X}'$ be measurable. Define
(for a given compact set $W\subset\R^d$), the
mapping $\Gamma_{f,W}\colon \bN(\BX)\to\bN(\BX')$ by
\begin{align}\label{eGammaop2}
\Gamma_{f,W}(\mu)(\cdot):=\int\I\{f(x) \in \cdot,x\in W\}g(x,\mu-\delta_x)\,\mu(\mathrm{d}x),\quad\mu\in\bN(\BX).
\end{align}
Let $\nu$ be a Poisson process as in Theorem \ref{tpoapprox}
and let $\nu_f$ be the (finite) Poisson process on $\BX'$ given by
$\nu_f:=\nu(f^{-1}(\cdot))$. Then $\kr(\Gamma_{f,W}(\xi),\nu_f)$
can still be bounded by the right-hand side of \eqref{epoa}.
\end{remark}

\begin{proof}[Proof of Theorem \ref{tpoapprox}] Assume for each $(x,r)\in\BX$ that
$\chi^{x,r}$ is a point process with the Palm distribution
$\BP^{\Gamma(\xi)|\Gamma(\xi)}_{(x,r)}$
and such that $(\omega,x,r)\mapsto \chi^{x,r}(\omega)$ is
measurable. (Later we shall choose $\chi^{x,r}$ in a specific way.)
From \cite[Theorem 3.1]{BSY21} we have that
\begin{align}\label{eBB}\notag
\kr&(\Gamma(\xi)_W,\nu)\\
&\le 
\|\BE[\Gamma(\xi)_W]-\BE[\nu]\|
+2\int \BE\|\Gamma(\xi)_W-(\chi^{x,r}_W-\delta_{(x,r)})\|\,
\BE[\Gamma(\xi)_W](\mathrm{d}(x,r)).
\end{align}
We now turn to the integral on
the right-hand side of \eqref{eBB}. It equals
\begin{align*}
I:=\int \BE\|\Gamma(\xi)_W-(\chi^{x,r}_W-\delta_{(x,r)})\|\,
\BE[\Gamma(\xi)_W](\mathrm{d}(x,r)).
\end{align*}
Obviously,
\begin{align}\label{e8.8}
&I
\le \iint  \I \{S(x) \cap S(y) \neq \emptyset\} \BE[\Gamma(\xi)](\mathrm{d}(y,s))\,
\BE[\Gamma(\xi)](\mathrm{d}(x,r))\\ \notag
&\quad +\iint  \I \{S(x) \cap S(y) \neq \emptyset\} \BE[\chi^{x,r}-\delta_{(x,r)}](\mathrm{d}(y,s))\,
\BE[\Gamma(\xi)](\mathrm{d}(x,r))\\ \notag
&\quad +\BE\iint  \I \{S(x) \cap S(y) = \emptyset\} |\Gamma(\xi)_W-(\chi^{x,r}_W-\delta_{(x,r)})|(\mathrm{d}(y,s))\,
\BE[\Gamma(\xi)](\mathrm{d}(x,r)), \notag
\end{align}
where here and in the following the integration with respect to $x$ and $y$ is always restricted to $W$.  
Using the GNZ equations, we find that the first term on the above right-hand side is given by
\begin{align*}
 \iint  \BE [g(x,r,\xi) \kappa(x,r,\xi)]\, \BE [g(y,s,\xi) \kappa(y,s,\xi)]
\I \{S(x) \cap S(y) \neq \emptyset\}\,\BQ^2(\mathrm{d}(r,s))\,\mathrm{d}(x,y).
\end{align*}

Since $\chi^{x,r}$ is a point process with the Palm distribution
  $\BP^{\Gamma(\xi)|\Gamma(\xi)}_{(x,r)}$, we obtain from
  \eqref{ePalm} and \eqref{eGNZmulti} that the second term on the
right-hand side of \eqref{e8.8} is given by
\begin{align*}
	& \iint  \BE [g(x,r,\xi+\delta_{(y,s)}) \kappa(x,r,\xi+\delta_{(y,s)})g(y,s,\xi+\delta_{(x,r)}) 
\kappa(y,s,\xi)]\\
	&\quad \quad \times \I \{S(x) \cap S(y) \neq \emptyset\} \,\BQ^2(\mathrm{d}(r,s))\,\mathrm{d}(x,y).
\end{align*}

To treat the third term on the right-hand side of \eqref{e8.8}
we use disagreement coupling. 
Assume for each $(x,r)\in\BX$ that
$\xi^{x,r}$ is a point process whose distribution is the Palm distribution
$\BP^{\xi|\Gamma(\xi)}_{(x,r)}$
and such that $(\omega,x,r)\mapsto \xi^{x,r}(\omega)$ is
measurable. It is a straightforward task to check that
\begin{align*}
\chi^{x,r}\overset{d}{=}\Gamma(\xi^{x,r}),\quad (x,r)\in\BX.
\end{align*}
Let $U(x):=(W+S) \setminus R(x)$. Note that since $R \subset S$ and $o \in S$, 
every pair of points $x,y \in W$ with $S(x) \cap S(y)=\emptyset$ satisfies $y \in U(x)$.
By the DLR equations \eqref{edlr}, Lemma \ref{lpalm} and 
\eqref{eygt1} 
we have that
\begin{align}
	\BP(\xi^{x,r}_{U(x)}\in\cdot \mid \xi^{x,r}_{U(x)^c})
\end{align}
is the distribution of a Gibbs process with PI
$\kappa_{U(x)\times \mathbb{Y},\xi^{x,r}_{U(x)^c}+\delta_{(x,r)}}$,
while
\begin{align}
	\BP(\xi_{U(x)}\in\cdot \mid \xi_{U(x)^c})
\end{align}
is the distribution of a Gibbs process with PI
$\kappa_{U(x)\times \mathbb{Y},\xi_{U(x)^c}}$. By Lemma \ref{lthreecopling}
there exist point processes $\eta$, $\tilde{\xi}$ and $\tilde{\xi}^{x,r}$ 
such that $\eta$ is a Poisson process on $\BX$ with intensity measure
$\alpha \lambda_d\otimes\BQ$, $\tilde\xi\overset{d}{=}\xi$,
$\tilde\xi^{x,r}\overset{d}{=}\xi^{x,r}$ 
and $\tilde{\xi} \le \eta$,
$\tilde{\xi}^{x,r}\le \eta$ almost surely.
Let $\Phi$ be an independent marking of $\eta$
with the marks uniformly distributed on the interval $[0,\alpha]$.
Then $\Phi$ is a Poisson process on $\BX\times[0,\alpha]$, whose intensity
measure is the product of $\lambda$ and Lebesgue measure on $[0,\alpha]$.
We now consider the disagreement coupling
\begin{align*}
	\psi:=T_{U(x)\times\BY,\tilde\xi_{U(x)^c}}(\Phi_{U(x)\times\BY\times[0,\alpha]}),\quad
	\psi^{x,r}:=T_{U(x)\times\BY,\tilde\xi^{x,r}_{U(x)^c}+\delta_{(x,r)}}(\Phi_{U(x)\times\BY\times[0,\alpha]}),
\end{align*}
where we refer to the notation introduced in the first
paragraph of Section \ref{secdisagreement}.
From Theorem \ref{tembedding}, Remark \ref{r22} and the DLR equations we
find that
\begin{align}
	\xi \overset{d}{=}\psi+\tilde{\xi}_{U(x)^c},\quad 
	\xi^{x,r}\overset{d}{=} \psi^{x,r}+\tilde\xi^{x,r}_{U(x)^c}.
\end{align}
Using this coupling we can rewrite the third term 
on the right-hand side of \eqref{e8.8}, $I_3$ say, as
\begin{align*}
	I_3=\BE\iint  \I \{S(x) \cap S(y) = \emptyset\} |\Gamma(\psi+\tilde{\xi}_{U(x)^c})-
	\Gamma(\psi^{x,r}+\tilde\xi^{x,r}_{U(x)^c})|(\mathrm{d}(y,s))\BE[\Gamma(\xi)](\mathrm{d}(x,r)).
\end{align*}
Since a point $y\in  W$ with $S(x) \cap S(y)=\emptyset$ 
satisfies $y \in U(x)$, it can only contribute to the total variation
$\|\Gamma(\psi+\tilde{\xi}_{U(x)^c})-
\Gamma(\psi^{x,r}+\tilde\xi^{x,r}_{U(x)^c})\|$
if $(y,s)\in \psi+\psi^{x,r}$ and if one of the following three
cases occurs. In the first case we have that $(y,s)\notin \psi^{x,r}$ and 
$g(y,s,\psi+\tilde{\xi}_{U(x)^c})=1$. By \eqref{eygt1} and the definition of $U(x)$
we have that  $g(y,s,\psi+\tilde{\xi}_{U(x)^c})=g(y,s,\psi)$.
The second case is $(y,s)\notin \psi$ and 
$g(y,s,\psi^{x,r})=1$. The third case is $(y,s)\notin|\psi-\psi^{x,r}|$
and $|g(y,s,\psi)-g(y,s,\psi^{x,r})|=1$.
By \eqref{eygt1}, the third case can only occur if
$|\psi-\psi^{x,r}|(R(y)\times \mathbb{Y})\ne 0$. The latter inequality holds
in the other two cases as well. Therefore,
\begin{align}\label{e3100}\notag
&\int  \I \{S(x) \cap S(y) = \emptyset\} |\Gamma(\psi+\tilde{\xi}_{U(x)^c})-
\Gamma(\psi^{x,r}+\tilde\xi^{x,r}_{U(x)^c})|(\mathrm{d}(y,s))\\
	&\le \int \I \{S(x) \cap S(y) = \emptyset\} \I\{|\psi-\psi^{x,r}|(R(y)\times \mathbb{Y})\ne 0\}\,
	(\psi+\psi^{x,r})(\mathrm{d}(y,s)).
\end{align}
By Theorem \ref{tdis}, $\psi+\psi^{x,r}$ is dominated by the
restriction of $\eta$ to $U(x)\times\BY$. Moreover, if
$|\psi-\psi^{x,r}|(R(y) \times \mathbb{Y})\ne 0$, then there exists
$(z,u)\in \eta$ with $z\in R(y)$ such that $(z,u)$ is
connected via $\eta^!_{(z,u)}:=\eta-\delta_{(z,u)}$ to
$\tilde\xi^{x,r}_{U(x)^c}+\tilde\xi_{U(x)^c}$. Since
$\tilde{\xi} \le \eta$, $\tilde{\xi}^{x,r}\le \eta$
we have that
$\text{supp}(\tilde\xi^{x,r}+\tilde\xi) \subseteq
\text{supp}(\eta)$ almost surely. This implies that $(z,u)$
is connected via
$\eta^!_{(z,u)}$ to
$\eta_{U(x)^c}$ and, 
hence, that $R(y)\xleftrightarrow{\eta^!_{(y,s)}}U(x)^c$.  Therefore we can bound the right-hand
side of \eqref{e3100} by
\begin{align*}
\int  \I \{S(x) \cap S(y) = \emptyset\}
\I\{R(y)\xleftrightarrow{\eta^!_{(y,s)}}U(x)^c\}\,
	\eta(\mathrm{d}(y,s)).
\end{align*}
Taking the expectation and using the Mecke formula yields
\begin{align*}
	&\BE \int  \I \{S(x) \cap S(y) = \emptyset\} |\Gamma(\psi+\tilde{\xi}_{U(x)^c})-
	\Gamma(\psi^{x,r}+\tilde\xi^{x,r}_{U(x)^c})|(\mathrm{d}(y,s))\\
	&\le \alpha \iint \I \{S(x) \cap S(y) = \emptyset\}
	\BP(R(y)\overset{\eta}{\longleftrightarrow}U(x)^c)\,\BQ(\mathrm{d}s)\,\mathrm{d}y.
\end{align*}
Therefore 
\begin{align*}
	I_3&\le \alpha \iint  \I \{S(x) \cap S(y) = \emptyset\}
	\BP(R(y)\overset{\eta}{\longleftrightarrow}(W+S)^c \cup R(x))\\
	&\quad \times  \BE [g(x,r,\xi) \kappa(x,r,\xi)]\, \mathbb{Q}(\mathrm{d}r)\,	\mathrm{d}(x,y).
\end{align*}
This finishes the proof.
\end{proof}

\bigskip

Next we provide a lemma which helps us control the
term $T_3$ on the right-hand side of \eqref{epoa} in the case
where $R(x)$ and $S(x)$ are balls; see Corollary \ref{cpoapprox}. 
The  lemma bounds  the one-arm
probabilities, a well-studied object in the theory of continuum
percolation (see e.g. \cite{MeesterRoy}). It is convenient to
introduce a random variable $Y$ with distribution $\BQ$ which is
independent of $\eta$. Let $o$ be the origin in $\mathbb{R}^d$. Then $C((o,Y),\eta+\delta_{(o,Y)})$ can be
interpreted as the cluster containing the {\em typical point} of
$\eta$, at least if the relation $\sim$ is {\em translation
  invariant}.  By this we mean that for any given $(x,r),(y,s)\in\BX$
we have that $(x,r)\sim (y,s)$ implies $(x+z,r)\sim (y+z,s)$ for each
$z\in\R^d$.

\begin{lemma}\label{lytpical} Let $\eta$ be a Poisson process on $\BX$ with intensity measure
$\alpha \lambda_d\otimes\BQ$. Assume that $\sim$ is translation invariant and let $u,v>0$. 
Then
\begin{align}
		\BP(B(o,u) \xleftrightarrow{\eta}B(o,u+v)^c)
		&\le \alpha \lambda_d(B(o,1))u^d\,	
\BP((o,Y) \xleftrightarrow{\eta+\delta_{(o,Y)}} B(o,v)^c).
\end{align}
\end{lemma}
\begin{proof} We have
\begin{align*} 
	&\I\{B(o,u) \xleftrightarrow{\eta}B(o,u+v)^c)\}\\
	&\quad \le \int \I\{x\in B(o,u), C(x,r,\eta)(B(o,u+v)^c\times\BY)>0\}\,\eta(\mathrm{d}(x,r)). 
\end{align*}
Taking expectations and using the Mecke equation yields
\begin{align*} 
	p:=\BP&(B(o,u) \xleftrightarrow{\eta}B(o,u+v)^c)\\
	&\le \alpha \iint \I\{x\in B(o,u)\} \BP(C(x,r,\eta+\delta_{(x,r)})(B(o,u+v)^c\times \BY)>0)
	\,\mathrm{d}x\,\BQ(\mathrm{d}r). 
\end{align*}
By translation invariance of $\sim$,
\begin{align*}
	C(x,r,\eta+\delta_{(x,r)})(B(o,u+v)^c\times \BY)
	=C(o,r,\theta_x\eta+\delta_{(o,r)})(B(-x,u+v)^c\times \BY),
\end{align*}
where $\theta_x\eta$ is the point process on $\BX$ defined by
$\theta_x\eta(B\times C):=\eta((B+x)\times C)$ for measurable
sets $B\subset\R^d$ and $C\subset\BY$. Since
$\theta_x\eta\overset{d}{=}\eta$ we obtain that
\begin{align*}
	p\le \alpha \int \I\{x\in B(o,u)\} \BP(C(o,Y,\eta+\delta_{(o,Y)})(B(-x,u+v)^c\times \BY)>0)
	\,\mathrm{d}x.
\end{align*}
If $x\in  B(o,u)$ then $B(o,v)\subset B(-x,u+v)$. 
This implies the asserted inequality.
\end{proof}

\begin{example}\label{exBoolean} \rm Let $R\in(0,\infty)$ and
	assume that $\BY=[0,R]\times\BY'$ for some 
	complete separable metric space $\BY'$. For
	$(x,r,w)$ we interpret $x$ as the center of a ball with radius $r$
	and mark $w$.  Define the translation invariant
	relation $\sim$ on $\R^d\times \BY$ by
	$(x,r,w)\sim (y,s,z)$ if $\|x-y\|\le r+s$ and $w\sim' z$, where
	$\sim'$ is a measurable symmetric  relation on $\BY'$; 
	see also \eqref{esim}.  Then $\bar\eta:=\eta(\cdot\times\BY')$ is a
	Poisson process on $\R^d\times \R_+$ with intensity measure
	$\alpha\lambda_d\otimes \bar\BQ$, where
	$\bar\BQ:=\BQ(\cdot\times\BY')$.  The associated {\em Boolean model}
	is given as
	\begin{align}
		Z:=\bigcup_{(x,r)\in\bar\eta}B(x,r).
	\end{align}
	Let us define the {\em critical intensity}
	\begin{align}\label{alcrit}
		\alpha_c:=
		\sup\{\beta>0:\int \Pi_{\beta\lambda_d\otimes \bar\BQ}(\{\mu\in\bN(\R^d\times\R_+):|C(o,r,\mu)|=\infty\})
		\,\bar \BQ (\mathrm{d}r)=0\},
	\end{align}
	where $\Pi_{\beta\lambda_d\otimes \bar\BQ}$ for $\beta>0$ is the distribution of a
	Poisson process with intensity measure $\beta\lambda_d\otimes \bar\BQ$
	and the cluster $C(o,r,\mu)$ is defined with respect to the relation \eqref{esim}.  
	This example can be generalized by replacing $[0,R]$ by the space 
	of all particles contained in some fixed ball; see Example \ref{expartic}. 
\end{example}

Assume that
\begin{align}\label{csharp}
	\BP((o,Y) \xleftrightarrow{\eta+\delta_{(o,Y)}} B(o,v)^c) \le c_1 e^{-c_2v},\quad v>0,
\end{align}
for constants $c_1,c_2>0$ (depending only on the dimension and on
$\alpha$). This implies that the graph $G(\eta)$ (introduced at the
beginning of this section) does not {\em percolate}, that is, it does
not have an unbounded connected component. (We refer to
\cite{MeesterRoy} for an extensive discussion of some standard models
of {\em continuum percolation}.) 

\begin{remark}\label{rperc}\rm
Consider the setting of
Example \ref{exBoolean} and assume that $\alpha<\alpha_c$.
Then we obtain from \cite[(3.7)]{Zie17}
 that
\eqref{csharp}
holds.
\end{remark}

\begin{corollary}\label{cpoapprox} Let $\xi$ be a Gibbs process on $\R^d\times\BY$
	with a PI $\kappa$ satisfying (Loc1) and
	(Dom2) for some $\alpha>0$ and assume that \eqref{csharp} holds. Define 
	$\Gamma$ by \eqref{eGammaop}, where $g$ is assumed to satisfy \eqref{eher}. 
Moreover, assume that there is some $u>0$ such that
	\begin{align}
		g(x,r,\mu)&=g(x,r,\mu_{B(x,u)}),
	\end{align}
        for all $(x,r,\mu)\in\R^d\times\BY\times\bN(\BX)$. Let
        $W\subset \R^d$ be a compact set and let $\nu$ be a
        Poisson process with finite intensity measure
        $\mathbb{E}[\nu]$. Then we have for all $v>0$
	\begin{align}\label{cpoa}\notag
   \kr(\Gamma(\xi)_W,\nu)
       \le 
\|&\BE[\Gamma(\xi)_W]-\BE[\nu]\| 
          + F_1(u,v)+F_2(u,v)\\
&+\I\{\mathrm{diam}(W)> 2(u+v) \}c_1 \BE[\Gamma(\xi) (W)] \lambda_d(W) u^d e^{-c_2v},
	\end{align}
	where
	\begin{align*}
		F_1(u,v):= &2\iint \I\{x,y\in W\} \BE [g(x,r,\xi) \kappa(x,r,\xi)]\, 
\BE [g(y,s,\xi) \kappa(y,s,\xi)]\,\\
		&\quad \quad \times\I \{\|x-y\| \le 2(u+v) \}\, \BQ^2(\mathrm{d}(r,s))\,\mathrm{d}(x,y),\\
		F_2(u,v):= &2\iint \I\{x,y\in W\} \BE [g(x,r,\xi+\delta_{(y,s)}) 
\kappa(x,r,\xi+\delta_{(y,s)})g(y,s,\xi+\delta_{(x,r)}) \kappa(y,s,\xi)]\\
		&\quad \quad \times\I\{\|x-y\| \le 2(u+v) \} \, \,\BQ^2(\mathrm{d}(r,s))\,\mathrm{d}(x,y).
	\end{align*}
The constants $c_1,c_2>0$ depend only on the dimension and on $\alpha$.
\end{corollary}

\begin{proof}
We apply Theorem \ref{tpoapprox} with $R:=B(o,u)$ and $S:=B(o,u+v)$ for some $u,v>0$. 
 We wish to bound $T_3$. Let $x,y\in W$ such
	that $S(x)\cap S(y)=\emptyset$, i.e.\ $\|x-y\|>2(u+v)$.
	If $B(y,u)\overset{\eta}{\longleftrightarrow}B(x,u)$
	then $B(y,u)\overset{\eta}{\longleftrightarrow}B(y,u+v)^c$.
	And if $B(y,u)\overset{\eta}{\longleftrightarrow}(W+S)^c$
	we can  use $y\in W$ to conclude that
	$B(y,u)\overset{\eta}{\longleftrightarrow}B(y,u+v)^c$. Hence we obtain from stationarity 
	\begin{align*}
		\BP(R(y)\overset{\eta}{\longleftrightarrow}(W+S)^c \cup R(x))
		\le \BP(B(o,u)\overset{\eta}{\longleftrightarrow}B(o,u+v)).
	\end{align*}
Moreover, since $\|x-y\|>2(u+v)$ and $x,y\in W$ we have $\mathrm{diam}(W)>2(u+v)$. By assumption \eqref{csharp} we obtain 
from Lemma \ref{lytpical} 
that $T_3$ is bounded by 
\begin{align*}
	&\I\{\mathrm{diam}(W)>2(u+v)\}c_1 u^d e^{-c_2v} \iint \I\{x,y\in W\} \BE [g(x,r,\xi) \kappa(x,r,\xi)] 
	\mathbb{Q}(\mathrm{d}r)\mathrm{d}(x,y)\\
	&\quad=\I\{\mathrm{diam}(W)>2(u+v)\}c_1  \BE[\Gamma(\xi) (W)] \lambda_d(W) u^d e^{-c_2v}  
\end{align*}
for some constants $c_1,c_2>0$. This proves the corollary.
\end{proof}

\begin{remark}\label{r9.8} \rm Assume that the probability measure
    $\mathbb{Q}$ in the Strauss process, the continuum random cluster
    model or the Widom-Rowlinson model is supported on a bounded
    subset of $[0,\infty)$. Then we can apply Remark \ref{rperc} to
    see that Corollary \ref{cpoapprox} holds for $\alpha<\alpha_c$,
    where $\alpha_c$ is the percolation threshold of a spherical
    Boolean model with radius distribution $\mathbb{Q}$.
Corollary \ref{cpoapprox} also holds for the area interaction process if additionally
    $\mathbb{Q}([r_1,\infty))=1$ for some $r_1>0$ (so that (Dom2)
    holds).

    For a Gibbs process with a pair potential, assume that $\BX=\R^d$
    and that $U(x,y)=U^*(\|x-y\|)$ for some measurable
    $U^*\colon \R_+\to [0,\infty]$.  Then (Dom2) holds and
in order to apply Theorem \ref{tpoapprox} we assume that
there exists $r_0\ge 0$ such that $U^*(r)=0$ for all
    $r\ge r_0$.  We can then apply Remark  \ref{rperc} 
    with $\BQ=\delta_{r_0/2}$ to see that Corollary \ref{cpoapprox}
    holds for $\alpha<\alpha_c$, where $\alpha_c$ is the percolation
    threshold of a spherical Boolean model with deterministic radius
    $r_0/2$.
\end{remark}

\section{Mat\'{e}rn type I thinnings of Gibbs processes}\label{sMatern}

In this section we apply Theorem \ref{tpoapprox} to scaled Mat\'{e}rn
type I thinning of a Gibbs process $\xi$ on $\R^d\times\mathbb{Y}$. 
As in Example \ref{exBoolean} we take $\mathbb{Y}:=[0,R] \times \mathbb{Y}'$
equipped with a probability measure $\mathbb{Q}$
and the relation $\sim$ defined there.
We assume that the PI $\kappa$ satisfies (Loc1) and
(Dom2) for some $\alpha>0$. Assume that
\begin{align}\label{kappain}
	\kappa(x,r,0)=\kappa(o,r,0),\quad x \in \mathbb{R}^d,\,r \in \mathbb{Y}.
\end{align}
Condition \eqref{kappain} holds for Gibbs 
processes with pair potential (see Remark \ref{r9.8}), for the Strauss process, the area interaction process, 
the continuum random cluster model and the Widom-Rowlinson model. 

Let $c>0$.  We can find for each
$n\in\N$ and each $x \in [0,n^{1/d}]^d$ a number $u_n(x)$ such that  
\begin{align} \label{hcm1}
	\sup_{x \in [0,1]^d} |c-n\mathbb{P}(\xi(B(n^{1/d}x,u_n(n^{1/d}x))\times \BY)=0)|=0
\end{align}
for all $n$ large enough. In fact, we can choose
\begin{align*}
	u_n(x):=\sup\{u\ge 0:\BP(\xi(B(x,u)\times\BY)=0)\ge c/n\}
\end{align*}
with the convention $\sup \emptyset:=0$. Since the measure $\BE\xi(\cdot\times\BY)$ 
is absolutely continuous (it is dominated by $\alpha\lambda_d$),
the probability $\BP(\xi(B(x,u)\times\BY)=0)$ is a continuous function of $u$.
(The probability that $\xi(\cdot\times\BY)$ has mass on the boundary
of a ball vanishes.)
Hence, it
follows straight from the definition that
\eqref{hcm1} holds, provided that $c< n$.

We assert that
\begin{align}
	\limsup_{n \to \infty} \frac{1}{\log n}\sup_{x \in [0,1]^d} u_n(n^{1/d}x)^d < \infty,
	\quad \liminf_{n \to \infty} \frac{1}{\log n} \inf_{x \in [0,1]^d} u_n(n^{1/d}x)^d > 0. \label{est:un}
\end{align}
Indeed, the first relation follows from
\eqref{hardlowbou} while the second can be derived from
Theorem \ref{tempty1} as in Example \ref{expartic}.

We consider the process
\begin{align} \label{defchi} \chi_n:=\sum_{(x,r) \in \xi} \I \{x \in
	[0,n^{1/d}]^d\} \I \{(\xi-\delta_{(x,r)})(B(x,u_n(x)) \times
	\mathbb{Y})=0\} \delta_{(n^{-1/d}x,r)}.
\end{align}

\begin{theorem} \label{hardtpo} Let $\xi$ be a 
	Gibbs process on $\R^d\times [0,R] \times \mathbb{Y}'$ with PI
	$\kappa$ satisfying (Loc1) and (Dom2) for some
	$\alpha \in (0,\alpha_c)$, where $\alpha_c$ is introduced at
	\eqref{alcrit}. Assume that \eqref{kappain} holds. 
	Let $c>0$ and let
	$\nu$ be a Poisson process on $\mathbb{R}^d \times \mathbb{Y}$ with
	intensity measure
	$$
	c \kappa(o,r,0) \I\{x \in [0,1]^d\}\,\mathrm{d}x
	\,\mathbb{Q}(\mathrm{d}r).
	$$ 
Define $a:=\liminf_{n \to \infty} \frac{1}{\log n} \inf_{x \in [0,1]^d} u_n(n^{1/d}x)^d$
and let $M>0$ satisfy
\begin{align}\label{defM}
		M< \min\Big(1,2^{-d}a e^{-\alpha \kappa_d 2^d R^d} 
\int \kappa(o,r,0) \mathbb{Q}(\mathrm dr)\Big).
		\end{align}
		Then $\kr(\chi_n,\,\nu)  \le n^{-M}$ for $n \in \mathbb{N}$ large enough.
\end{theorem}
\begin{proof}
We apply Corollary \ref{cpoapprox} with $W,u,v$ depending on $n$. Let $W_n:=[0,n^{1/d}]^d$,
		$u_n:= \sup_{x \in [0,1]^d} u_n(n^{1/d}x)$ and
		$v_n:=2c_2^{-1} \log n$, where $c_2$ is the constant from Corollary
		\ref{cpoapprox}. We let $g$ and $f$ (in Remark \ref{remf}) depend on $n$ as well. For
		$(x,r,\mu) \in \mathbb{R}^d\times \mathbb{Y} \times \mathbf{N}$ let
		\begin{align}
			g_n(x,r,\mu):= g_n(x,\mu):= \I\{\mu(B(x,u_n(x))\times \BY)=0\}
		\end{align}
		and $f_n(x,r):=(n^{-1/d}x,r)$. Obviously, $g_n$ satisfies the
	hereditary property \eqref{eher} as well as 
	\eqref{eygt1}. 
	As Remark
	\ref{rperc} shows, \eqref{csharp} holds.
	
	For all $n$ with $\inf_{x \in [0,1]^d} u_n(n^{1/d}x) >2R$ we
        find from (Loc1) and \eqref{kappain} that for all
        $B \in \mathcal{X}$,
\begin{align*}
          \mathbb{E}[\chi_n](B)&=\iint \I\{(n^{-1/d}x,r) \in B \cap  ([0,1]^d \times \mathbb{Y})\} 
\mathbb{E} [g_n(x,r,\xi)] \kappa(x,r,0)\,\mathrm{d}x \,\mathbb{Q}(\mathrm{d}r)\\
                               &= n\iint \I\{(x,r) \in B \cap ([0,1]^d \times \mathbb{Y})\} \mathbb{E} [g_n(n^{1/d}x,r,\xi)] 
\kappa(o,r,0)\,\mathrm{d}x\,\mathbb{Q}(\mathrm{d}r).
\end{align*}
Hence, for those $n$ it holds that
\begin{align*}
		\|\mathbb{E}[\chi_n]-\mathbb{E}[\nu]\|
		\le \sup_{x \in [0,1]^d} 
		|c-n\mathbb{P}(\xi(B(n^{1/d}x,u_n(n^{1/d}x))\times \BY)=0)| \int  \kappa(o,r,0)\,\mathbb{Q}(\mathrm{d}r),
	\end{align*}
	which vanishes by \eqref{hcm1}.
	
	Since we assumed that (Dom2) holds for some constant $\alpha>0$, we find from \eqref{hcm1} 
	that for all $n$ large enough,
	\begin{align}
		\mathbb{E} [g_n(x,r,\xi) \kappa(x,r,\xi)] \le 2 \alpha c n^{-1}. \label{chinbou}
	\end{align}
	Hence, the term $F_1(u_n,v_n)$ in Corollary \ref{cpoapprox} is for those $n$ bounded by
	\begin{align}
		&8(\alpha c)^2 n^{-2}\int \I\{\|x-y\|\le 2u_n+4c_2^{-1}\log n\} \I\{x,y \in [0,n^{1/d}]^d\}\,  \mathrm{d}(x,y)\nonumber\\
		&\quad  \le \frac{8(\alpha c)^2 \kappa_d (2u_n+4c_2^{-1}\log n)^d}{n}.\label{F1bou}
	\end{align}

	For $F_2(u_n,v_n)$ we obtain the bound
	\begin{align*}
		&2\alpha ^2\iint  \mathbb{P}((\xi+\delta_{(y,s)})(B(x,u_n(x))\times \mathbb{Y})=0,\,(\xi+\delta_{(x,r)})(B(y,u_n(y))\times \mathbb{Y})=0)
\nonumber\\
		& \quad \times \I\{\|x-y\|\le 2u_n+4c_2^{-1}\log n\} \I\{x,y \in [0,n^{1/d}]^d\}\,\BQ^2(\mathrm{d}(r,s))\,\mathrm{d}(x,y).
	\end{align*}
	Note that the probability in the integrand of $F_2(u_n,v_n)$ is zero if $y \in B(x,u_n(x))$. For $y \in B(x,u_n(x))^c$ it is given by
	\begin{align}
		\mathbb{P}(\xi(B(x,u_n(x))\cup B(y,u_n(y)))=0). \label{hardCtint}
	\end{align}
	Since the ball $B(y+\frac{u_n(y)}{2}\frac{y-x}{\|y-x\|}, \frac{u_n(y)}{2})$ is
	contained in $B(y,u_n(y))$, \eqref{hardCtint} is
	bounded by
\begin{align}\label{hardPP}
& \mathbb{P}\Big(\xi(B(x,u_n(x)))=0,\xi\Big(B\Big(y+\frac{u_n(y)}{2}\frac{y-x}{\|y-x\|}, 
		\frac{u_n(y)}{2}\Big)\Big)=0\Big)\\ \notag
&\quad=\mathbb{E}\Big[g_n(x,r,\xi)\, 
\mathbb{P}\Big(\xi\Big(B\Big(y+\frac{u_n(y)}{2}\frac{y-x}{\|y-x\|}, 
		\frac{u_n(y)}{2}\Big)\Big)=0\, \Bigl| \, \xi_{B(y+\frac{u_n(y)}{2}\frac{y-x}{\|y-x\|}, 
\frac{u_n(y)}{2})^c}\Big)\Big].
	\end{align}
	Note that $\|y+\frac{u_n(y)}{2}\frac{y-x}{\|y-x\|}-x\|=\|y-x\|+\frac{u_n(y)}{2}$.  Hence, the balls
	$B(y+\frac{u_n(y)}{2}\frac{y-x}{\|y-x\|}, \frac{u_n(y)}{2})$ and
	$B(x,u_n(x))$ are disjoint for $y \in B(x,u_n(x))^c$. Therefore, we obtain from
	Theorem \ref{tempty1} (applied in the same way as in
	Example \ref{expartic}) that the conditional probability in
	\eqref{hardPP} is bounded from above by $e^{-c_0(u_n(y)/2-2R)^d}$
	with $c_0=\int\kappa(o,r,0)\,\mathbb{Q}(\mathrm{d}r)\,e^{-\alpha\kappa_d 2^d R^d }$ from
	\eqref{exPartc0}. Thus we conclude that
	\begin{align*}
		F_2(u_n,v_n) \le 2\alpha^2 c\kappa_d(2 u_n+4c_2^{-1}\log n)^d \sup_{y \in [0,1]^d}e^{-c_0(u_n(n^{1/d}y)/2-2R)^d}.
	\end{align*}
Recalling the definition of the constant $a$ in Theorem  \ref{hardtpo} we can now use
\eqref{est:un} to conclude that $F_2(u_n,v_n)$ is for $n \in \mathbb{N}$ large enough bounded by
	\begin{align}
		\beta_1 (\log n)^d \exp \big\{-c_02^{-d}a
\log n\big\} \label{F2bou}
	\end{align}
for some constant $\beta_1>0$ that does not depend on $n$.  By definition of $c_0$ this is bounded  
by $n^{-M}$ for the constant $M$ from \eqref{defM}.
	
	Finally, we bound the last term on the right-hand side of \eqref{cpoa}. From \eqref{chinbou} we obtain that 
$\mathbb{E}[\chi_n(W)]\le 2\alpha c$. Hence, the last term in \eqref{cpoa} is bounded by
	\begin{align}
		2cc_1 \alpha n  u_n^d \exp(-c_2 v_n) \le \frac{\beta_2 \log n}{n} \label{F3bou}
	\end{align}
	for some constant $\beta_2>0$ not depending on $n$, where we have used \eqref{est:un} again.
	
	Now the assertion follows from \eqref{hcm1} and the bounds in
        \eqref{F1bou}, \eqref{F2bou}, \eqref{F3bou}.  
\end{proof}

\noindent
{\bf Acknowledgments:} We wish to thank Steffen Betsch for making
several useful comments. 


\begin{thebibliography}{99}

\bibitem{BB92}
Barbour, A.D.\ and Brown, T.C.\ (1992). 
Stein's method and point process approximation. 
{\em Stoch.\ Proc.\ Appl.} {\bf 43}, 9--31.



\bibitem{BHLV20}
Benes, V., Hofer-Temmel, C., Last, G.\ and Vecera, J.\ (2020). 
Decorrelation of a class of 
Gibbs particle processes and asymptotic properties of U-statistics.
{\em J.\ Appl.\ Probab.} {\bf 57}, 928--955.

\bibitem{BetschLast22}
Betsch, S.\ and Last, G.\ (2022).
On the uniqueness of Gibbs distributions with a non-negative and subcritical pair potential.
To appear in {\em Ann. Inst. H. Poincar\'{e} Probab. Statist.}

\bibitem{BY14}
Blaszczyszyn, B.\ and Yogeshwaran, D.\ (2014).
On comparison of clustering properties of point processes.
{\em Adv. Appl. Probab.} {\bf 46}, 1--20.

\bibitem{BreMass96} 
Br\'emaud, P.\ and  Massouli\'e, L.\ (1996).
Stability of non-linear Hawkes processes.
{\em Ann.\ Probab.} {\bf 24}, 1563--1588.

\bibitem{BSY21} Bobrowski, O., Schulte, M. and Yogeshwaran, D.\ (2022). 
Poisson approximation under stabilization and Palm coupling. {\em Ann. Henri Lebesgue} {\bf 5,} 1489--1534.


\bibitem{CSKM13}
Chiu, S.N., Stoyan, D., Kendall, W.S.\ and Mecke, J.\ (2013).
{\em Stochastic Geometry and its Applications.}
3rd edn.\ Wiley, Chichester.



\bibitem{Dereudre09}
Dereudre, D.\ (2009).
The existence of quermass-interaction processes for nonlocally stable interaction and nonbounded convex grains.
{\em Electron. Commun. Probab.} {\bf 21}, 1--11. 
 
\bibitem{Dereudre18}
Dereudre, D.\ (2019).
Introduction to the theory of Gibbs point processes. 
in {\em Lecture Notes in Mathematics} {\bf 2237}, Stochastic Geometry, Chapter 5, CEMPI subseries.


\bibitem{DeGeDro12}
Dereudre, D., Georgii, H.O. \ and  Drouilhet, R. (2012).
Existence of Gibbsian point processes with geometry-dependent interactions.
{\em Prob. Theory Relat. Fields} {\bf 153}, 643--670. 

\bibitem{DeHou15}
Dereudre, D. \ and Houdebert, P. (2015).
Infinite volume continuum random cluster model.
{\em Electron. J. Probab.} {\bf 20}, 1--24. 


\bibitem{DVass19}
Dereudre, D. \ and Vasseur, T. (2020).
Existence of Gibbs point processes with stable infinite range interaction. 
{\em J.\ Appl.\ Probab.} {\bf 57}, 775--791. 


\bibitem{Georgii76}
Georgii, H.-O.\ (1976).
Canonical and grand canonical Gibbs states for continuum systems. 
{\em Commun. Math. Phys.} {\bf 48}(1), 31--51.

\bibitem{GeorHae96}
Georgii, H.-O.\ and H\"aggstr\"om, O.\ (1996).
Phase transition in continuum Potts models.
{\em Commun. Math. Phys.} {\bf 181}, 507--528.

\bibitem{GeorKun97}
Georgii, H.-O., K\"{u}neth, T. (1997). Stochastic order of point processes.
{\em J. Appl. Probab.} {\bf 34}, 868--881.

\bibitem{GeorYoo05}
Georgii, H.-O., Yoo, H.J. (2005).
Conditional intensity and Gibbsianness of determinantal point processes
{\em J. Stat. Phys.} {\bf 118}, 55-–84.

\bibitem{HTHou17}
Hofer-Temmel, C. \ and Houdebert, P. (2019). 
Disagreement percolation for Gibbs ball models. {\em 	Stoch. Proc. Appl.} {\bf 129}, 3922--3940.

\bibitem{HT19}
Hofer-Temmel, C.\ (2019).
Disagreement percolation for the hard-sphere model.
{\em Electron.\ J.\ Probab.} {\bf 24}, 1--22.

\bibitem{HoSoo13}
Holroyd, A.\ and Soo, T.\ (2013).
Insertion and deletion tolerance of point processes
{\em Electron.\ J.\ Probab.} {\bf 18}, 1--24.

\bibitem{Jansen19}
Jansen, S. (2019).
Cluster expansions for Gibbs point processes.
{\em Adv. Appl. Probab.} {\bf 51}, 1129--1178.

\bibitem{Kallenberg}
Kallenberg, O.\ (2002).
\newblock {\it Foundations of Modern Probability}. 
\newblock 2nd edn.\ Springer, New York.

\bibitem{Kallenberg17}
Kallenberg, O.\ (2017). 
\newblock {\it Random Measures, Theory and Applications.}  
\newblock Springer, Cham.



\bibitem{LastBrandt95}
Last, G. and Brandt, A. (1995). {\em Marked Point Processes on the
Real Line: The Dynamic Approach.}
Springer-Verlag, New York.

\bibitem{LPY2021}
Last, G., Peccati, G.\ and Yogeshwaran, D.\ (2023).
Phase transitions and noise sensitivity on the Poisson 
space via stopping sets and decision trees.
To appear in {\em Random Struct. Algor.}

\bibitem{LastPenrose17}
Last, G.\ and Penrose, M.\ (2017).
{\em Lectures on the Poisson Process.}
Cambridge University Press. 

\bibitem{Mase00}
Mase, S.\ (2000). Marked Gibbs processes and asymptotic
normality of maximum pseudo‐likelihood estimators. 
{\em Math.\ Nachr.} {\bf 209}, 151--169.

\bibitem{MaWaMe79}
Matthes, K.,  Warmuth, W. \ and Mecke, J.\ (1979).
Bemerkungen zu einer Arbeit von Nguyen Xuan Xanh und Hans Zessin.
{\em Math.\ Nachr.} {\bf 88}, 117--127.

\bibitem{Mecke67} 
Mecke, J.\ (1967). 
\newblock{Station\"are zuf\"allige Ma\ss e auf lokalkompakten Abelschen Gruppen}. 
{\em Z. Wahrsch. verw. Gebiete} {\bf 9}, 36--58.

\bibitem{MeesterRoy}
Meester, R.\ and Roy, R.\ (1996).
{\em Continuum Percolation.} Cambridge University Press, Cambridge.


\bibitem{MoeWaa07}
M\o{}ller, J. \ and Waagepetersen, R.P.\ (2007). Modern statistics for 
spatial point processes. {\em Scand. J. Statist.} {\bf  34}, 643-–684.

\bibitem{NgZe79}
Nguyen, X.X.\ and  Zessin, H.\ (1979).
Integral and differential characterizations of the Gibbs process. 
{\em Math.\ Nachr.} {\bf 88}, 105-115.

\bibitem{Otto20}
Otto, M.\ (2020).
Poisson approximation of Poisson-driven point processes and extreme values in stochastic geometry. {\em Preprint.}
arXiv:2005.10116.

\bibitem{Ruelle69}
Ruelle, D.\ (1969).
{\em Statistical Mechanics: Rigorous Results.}
Benjamin, New York.

\bibitem{Ruelle70}
Ruelle, D.\ (1970).
Superstable interactions in classical statistical mechanics.
{\em Commun.\ Math.\ Phys.} {\bf 18}, 127--159.

\bibitem{SchrYuk13}
Schreiber, T.\ and Yukich, J.E.\ (2013).
Limit theorems for geometric functionals of Gibbs point processes.
{\em Ann. Inst. H. Poincar\'{e} Probab. Statist.} {\bf 49}, 1158--1182.

\bibitem{Schuh09}
Schuhmacher, D.\ (2009).
Distance estimates for dependent thinnings of point processes with densities.
{\em Electron. J. Probab.} {\bf 14}, 1080--1116.

\bibitem{SchuhStucki13}
Schuhmacher, D.\ and Stucki, K.\ (2013).
Bounds for the probability generating functional of a Gibbs point process.
{\em Adv. Appl. Prob.} {\bf 46}, 21--34.

\bibitem{BergMaes94}
Van den Berg, J.\ and Maes, C.\ (1994). Disagreement percolation in the study of
Markov fields. {\em Ann. Probab.} {\bf 22}, 749--763. 

\bibitem{WidRow70}
Widom, B. \ and Rowlinson, J. S. \ (1970).
New model for the study of liquid-vapor phase transitions.  {\em J. Chem. Phys.} {\bf 52}, 1670--1684. 

\bibitem{Zie17}
Ziesche, S.\ (2018).
Sharpness of the phase transition and lower bounds for the
critical intensity in continuum percolation on $\R^d$.
{\em Ann. Inst. H. Poincar\'{e} Probab. Statist.} {\bf 54}, 866--878.



\end{thebibliography}
\end{document}